\documentclass[11pt]{amsart}
\usepackage{latexsym,cite, amssymb,amsmath , amsthm , amsfonts , ifpdf}
\usepackage{color}
\usepackage{graphicx}
\usepackage{multicol}
\usepackage{verbatim}
\usepackage[utf8x]{inputenc}
\usepackage{tikz}

\usepackage{amsmath}
\usepackage{amssymb}

\usetikzlibrary{arrows , shapes , trees , backgrounds}
\usetikzlibrary{intersections}

\makeatletter
\let\NaT@parse\undefined
\makeatletter

\usepackage [ colorlinks, citecolor=blue, anchorcolor=blue, linkcolor=blue ]{ hyperref }

\theoremstyle{plain}
\newtheorem{theorem}{Theorem}[section]
\newtheorem*{theorem*}{Theorem}

\newtheorem{pro}[theorem]{Proposition}
\newtheorem{Def}[theorem]{Definition}
\newtheorem{lem}[theorem]{Lemma}

\newtheorem{prop}[theorem]{Proposition}

\theoremstyle{definition}
\newtheorem{definition}[theorem]{Definition}
\newtheorem*{Def*}{Definition}

\newtheorem{remark}[theorem]{Remark}

\numberwithin{equation}{section}
\newcommand{\bpo}{\begin{pro}}
	\newcommand{\epo}{\end{pro}}
\newcommand{\be}{\begin{equation}}
	\newcommand{\ene}{\end{equation}}
\newcommand{\br}{\begin{remark}}
	\newcommand{\er}{\end{remark}}
\newcommand{\bl}{\begin{lem}}
	\newcommand{\el}{\end{lem}}
\newcommand{\bd}{\begin{Def}}
	\newcommand{\ed}{\end{Def}}
\newcommand{\ben}{\begin{enumerate}}
	\newcommand{\een}{\end{enumerate}}
\newcommand{\bp}{\begin{proof}}
	\newcommand{\ep}{\end{proof}}
\newcommand{\beq}{\begin{equation*}}
	\newcommand{\eeq}{\end{equation*}}
\newcommand{\bear}{\begin{eqnarray*}}
	\newcommand{\eear}{\end{eqnarray*}}
\newcommand{\bt}{\begin{theorem}}
	\newcommand{\et}{\end{theorem}}
\newcommand{\bst}{\begin{split}}
	\newcommand{\est}{\end{split}}

\newcommand{\bal}{\begin{aligned}}
	\newcommand{\eal}{\end{aligned}}

\renewcommand{\P}{\partial}
\newcommand{\F}[2]{\frac{#1}{#2}}
\newcommand{\la}{\langle}
\newcommand{\ra}{\rangle}

\newcommand{\R}{\mathbb{R}}

\newcommand{\bnb}{\bar{\nabla}}
\newcommand{\nb}{\nabla}

\newcommand{\ex}{e^{-\F{|\vec{F}|^{2}}{4}}}

\newcommand{\Cd}{covariant derivative}

\newcommand{\wrt}{with respect to}
\newcommand{\Sc}{\varepsilon}

\newcommand{\vp}{\varphi}

\newcommand{\Lp}{Lipschitz}

\def\XXint#1#2#3{{\setbox0=\hbox{$#1{#2#3}{\int}$}
		\vcenter{\hbox{$#2#3$}}\kern-.5\wd0}}

\def\XXint#1#2#3{{\setbox0=\hbox{$#1{#2#3}{\int}$}
        \vcenter{\hbox{$#2#3$}}\kern-.5\wd0}}

\begin{document}

	\title[Dirichlet problem]{The Dirichlet problem for the minimal surface system on smooth domains}
	\author{Caiyan Li and Hengyu Zhou}
	\address[Caiyan Li]{School of Mathematical Sciences, Xiamen University, Xiamen, 361005, P. R. China}
	\email{caiyanli@xmu.edu.cn}
	\address[Hengyu Zhou]{ College of Mathematics and Statistics, Chongqing University, Huxi Campus, Chongqing, 401331, P. R. China}
	\address{Chongqing Key Laboratory of Analytic Mathematics and Applications, Chongqing University, Huxi Campus, Chongqing, 401331, P. R. China}
	\email{zhouhyu@cqu.edu.cn}
	\subjclass[2010]{}
	\begin{abstract}
		In this paper, we propose a new assumption \eqref{condition:A} that involves a small oscillation and $C^2$ norms for  maps from  smooth bounded domains into Euclidean spaces. Furthermore, by assuming that the domain has non-negative Ricci curvature, we establish the Dirichlet problem for the minimal surface system via the mean curvature flow (MCF) with boundary. The long-time existence of such flow is derived using Bernstein-type theorems of higher codimensional self-shrinkers in the whole space and the half-space. Another novel aspect is that our hypothesis imposes no restriction on the diameter of the domains, which implies an existence result for an exterior Dirichlet problem of the minimal surface system.
	\end{abstract}
	\date{\today}

	\maketitle

	\section{Introduction}
	In differential geometry, minimal hypersurfaces and high codimension minimal submanifolds share many natural similarities but also exhibit essential differences for many classical problems. In this paper, we focus on one of those problems: the Dirichlet problem for the minimal surface system (MSS)
    	\begin{equation}\label{DP:MSS}
        \left\{
        \begin{aligned}
            g^{ij}(f)f^A_{ij}&=0\quad A=1,2,\cdots,m\text{\quad on~} E , (m\geq 2),\\
            f &=\psi \text{\quad on~}\P E ,
        \end{aligned}
        \right.
    \end{equation}
    which seeks higher codimensional minimal graphs with prescribed boundary data. For a more precise description we refer readers  to \eqref{DP:MSSB} in Section \ref{Sec pre}.

    In the codimension one case ($m=1$), \eqref{DP:MSS} (see also \eqref{DP:MSE}) is called the Dirichlet problem for the minimal surface equation. Jenkins and Serrin \cite{Jenkins-Serrin1968} first showed its existence and uniqueness for any Euclidean mean-convex domains with continuous boundary data. For arbitrary bounded domains, small oscillation conditions on the boundary data are necessary, as shown by Jenkins-Serrin \cite{Jenkins-Serrin1968} and Williams \cite{Williams1984}.

    In the higher codimensional case ($m\geq 2$), the situation becomes intricate due to the non-scalar nature of the solutions to \eqref{DP:MSS}. If the boundary data exhibits sufficiently large oscillations, Lawson and Osserman \cite{Lawson-Osserman1977} constructed numerous examples demonstrating the non-existence of solutions to  \eqref{DP:MSS} on Euclidean balls with $m\geq 3$. For recent advancements in this area, readers are referred to the work of Zhang \cite{Zhang2018} and Xu,Yang and Zhang \cite{XuYangZhang2019}. \\

    The first general existence result for \eqref{DP:MSS} on Euclidean convex domains was derived by Wang \cite{Wang2004}, based on higher codimensional mean curvature flows (MCF) in \eqref{MCF3}. The long-time existence of such flows at boundary requires a $C^{1,\alpha}$ gradient estimate near the boundary, as discussed by Thorpe \cite{Thorpe12} via the classical elliptic equation techniques.   Ding, Jost and Xin \cite{Ding-Jost-Xin2023} extended Wang's result to the case of Euclidean mean-convex domains. Their $C^{1,\alpha}$ gradient estimates \cite{DJX25} rely on Lieberman's sophisticated approach to parabolic equations \cite{Lieberman1996}. An interesting aspect is that both results in \cite{Wang2004} and \cite{Ding-Jost-Xin2023} impose requirements on the mean-convex property and diameter restrictions. For detailed explanations, see Theorem \ref{codimension:higher} and Remark \ref{remark:motivation}.

    In this paper, to solve \eqref{DP:MSS}, we propose a new oscillation condition that removes the mean-convex assumption by exploring the geometric properties near the boundary. Throughout this paper, any domain $E$ is assumed to have embedded and smooth boundary. We use the following notation: $\bar{E}=E\cup\partial E$, $E_\delta=\{x\in E: d(x,\P E)<\delta\}$, $w(\psi)=\max\limits_{1\leq A\leq m}\{\sup\limits_{\bar{E}}\psi^A -\inf\limits_{\bar{E}} \psi^A  \}$, and $|D^k\psi|$ denotes the norm of the $k$-th covariant derivative  of $\psi$. In particular,
    $ |D \psi| = \sup_{|\tau| = 1} |D\psi(\tau)|$ and $ |D^2 \psi| = \sup_{|\tau| = 1} |D^2 \psi(\tau, \tau)|$,
    where  $\tau$ ranges over unit tangent vectors of $E$.\\ \indent Our first main result is as follows.
	\begin{theorem}\label{main:thm:A}
Let $M^{n}$ be a complete Riemannian manifold with non-negative Ricci curvature and let $E\subset M$ be a bounded smooth domain. Then there exists a constant $\delta_0=\delta_0(n,\P E)>0$ such that for any $\delta\in(0,\delta_0)$ and any smooth map $\psi=(\psi^{1},\ldots,\psi^{m}):\bar{E} \rightarrow \R^m (m\geq 2)$ satisfying
		\begin{equation}\label{condition:A}
			\begin{split}
				\max\{\F{w
					(\psi)}{\delta}&+\sup_{E_{\delta}}|D\psi|+32n\delta \sup_{E_{\delta}}|D^2\psi|,\sup_{\bar{E}}|D\psi|\}<1,\\
			\end{split}
		\end{equation}
there exists a map $f$ from $\bar{E}$ to $\R^m$ in $C^\infty(E;\R^m)\cap C(\bar{E};\R^m)$  solving the  Dirichlet problem for the MSS \eqref{DP:MSS} with boundary data $\psi$. Moreover, $f$ is unique as a strictly length decreasing solution to \eqref{DP:MSS}.
	\end{theorem}
  The assumption \eqref{condition:A} ensures that the graph of the initial map $\psi$ is strictly length decreasing. The non-negative Ricci curvature assumption makes the evolution results of Tsui and Wang \cite{Tsui-Wang2004} still remaining valid in our setting. Consequently, the length decreasing property imposed by \eqref{condition:A} will be preserved along the MCF with boundary in Proposition \ref{prop:preserved} and Lemma \ref{first:step:MCF}.\\
  \indent Our long-time existence, which is in fact the $C^{1,\alpha}$ gradient estimate, of the MCF flow \eqref{MCF3} comes from White's regularity results \cite{White2005,White21}. In the proof of Lemma \ref{second:step:MCF} we show that the Gaussian density of the MCF flow \eqref{MCF3} is $\F{1}{2}$ at the boundary or is $1$ in its interior. In fact, we show that  the {\Lp} length decreasing self-shrinker on the whole space and the half-space is a linear space in Theorem \ref{st:A:self-shrinker} and Theorem \ref{lp:self-shrinker}, respectively. Those results have independent interests. \\
  \indent Finally the {\Lp} limit of the MCF flow \eqref{MCF3} is stationary and strictly length decreasing. By the smoothness result in \cite[Theorem 4.1]{Wang2004}, this limit is smooth, which gives the solution in Theorem \ref{main:thm:A}. Its uniqueness follows directly from Lee and Wang \cite[Theorem 3.2]{Lee-Wang2008}. \\
  \indent  Removing the mean-convex requirement is more natural because any constant map is automatically a solution to \eqref{DP:MSS}. In our setting, the boundary $\P E$ may have finite many components. Another new aspect is that $\delta_0$  only reflects the information in a fixed neighborhood of $\P E$ (see Remark \ref{remark:definitioneta}). Therefore, the exterior Dirichlet problem of the MSS can be considered. The following result is inspired by Simon \cite{Simon1983}, in which a more comprehensive characterization of its asymptotic behavior is established in the codimension one case.

	\begin{theorem}\label{main:thm:B}
	Suppose $E\subset\mathbb{R}^n$ is a smooth domain such that $\mathbb{R}^n\setminus E$ is bounded. Fix any constant $c\in (0,1)$. Then there exists a constant  $\tilde{\delta_0}=\tilde{\delta}_0(n,\P E)>0$ such that for any $\delta$ in $(0,\tilde{\delta}_0)$ and any smooth map $\psi:\bar{E} \rightarrow \R^{m}(m\geq 2)$ satisfying
		\begin{equation}\label{condition:B}
			\F{w(\psi)}{\delta}+\sup
			_{\bar{E}}|D\psi|+32 n\delta\sup_{\bar{E} }|D^2\psi|<1-c,
		\end{equation}
	there exists a solution $f:\bar{E}\rightarrow \R^m\in  C^{\infty}(E)\cap C(\bar{E})$ to the above Dirichlet problem of MSS \eqref{DP:MSS}. Moreover, $f$ satisfies the asymptotic behavior
		\begin{equation}
			\lim_{|x|\rightarrow \infty}|Df-l|=0,
		\end{equation}
		where $Df$ denotes the gradient of $f$ and $l$ is a constant vector in $\R^{n+m}$.
	\end{theorem}

The proof of Theorem \ref{main:thm:B} can be summarized as follows. Let $Q_r$ denote the intersection of $E$ and the open ball in $\R^{n}$ centered at the origin  with radius $r$. By applying Theorem \ref{main:thm:A} on $Q_r$ and letting $r\rightarrow \infty$, we can obtain the existence part of Theorem \ref{main:thm:B}.
The critical point is to show that for each domain $Q_r$, the constant $\tilde{\delta}_0(n,\P Q_r)$ depends only on $n$ and the boundary $\P E$ for sufficiently large $r$. The asymptotic behavior part follows from Bernstein-type theorems of strictly length decreasing high codimensional minimal graphs; see, for example \cite[Theorem 1.1]{Wang2003},\cite{Xin21}.

 In fact, our work is an effort to find elliptic versions of higher codimensional graphic mean curvature flows in \cite{Wang2002,Tsui-Wang2004}.

 This paper is organized as follows. In Section \ref{Sec pre} we collect some basic facts. In Section \ref{Sec bdry est} we establish the boundary gradient estimate for higher codimensional mean curvature flows, where we propose new oscillation conditions.  In Section \ref{Sec Bernstein} we present some Bernstein-type results of {\Lp} length decreasing self-shrinkers and minimal graphs in half-spaces. Finally, we prove Theorem \ref{main:thm:A} and  Theorem \ref{main:thm:B} in Section \ref{Sec pf thm:A} and Section \ref{Sec  pf thm:B}, respectively.

	\section{Preliminaries}\label{Sec pre}
  In this section we collect some facts for later use.
	\subsection{The Dirichlet problem}\label{Sec 2.1}
 Let $f=(f^1,\cdots, f^m)$ be a smooth map from $E$ to $\R^m$ and $\Gamma(f)$ denote the graph of $f$. Then $\Gamma(f)$ is an $n$-dimensional  submanifold with respect to the product metric on $E\times\R^m$. Fix a local coordinate system $\{x_1,\cdots, x_n\}$ on $E$. Then $\{\F{\P}{\P x_i}+df(\frac{\P }{\P  x_i}),i=1,\cdots, n\}$ is a local frame on $\Gamma(f)$. Denote $\la \F{\P}{\P x_i}, \F{\P}{\P x_j}\ra_{\sigma_M}$ by $\sigma_{ij}$. Here $\sigma_M$ denotes the metric of $M$. Therefore the induced metric on $\Gamma(f)$ can be written as
\begin{equation}\label{coe:g}
	g_{ij}(f)=\sigma_{ij}+\sum_{A=1}^{m}\frac{\P f^A}{\P x_i}\frac{\P f^A}{\P x_j}, \, \text{and}\, (g^{ij}(f))=(g^{ij}(f))^{-1}.
\end{equation}
\bl\label{lm:gij}With the assumption as above, the matrix $(g_{ij})$ in \eqref{coe:g} satisfies
\begin{equation}
    (1+|Df|^2)(\sigma_{ij})\geq (g_{ij})\geq (\sigma_{ij}),\ \text{and}\  (\sigma^{ij})\geq (g^{ij})\geq  \frac{1}{1+|Df|^2}(\sigma^{ij}).
\end{equation}
Here $(\sigma^{ij})=(\sigma_{ij})^{-1}$ and $|Df|=\sup_{|\tau|=1}|Df(\tau)|$.
\el
The proof will be deferred to the frame described in \eqref{relation:A}.

  Let $\bar{\nabla }$ be the connection on $E$. The mean curvature vector $\vec{H}$ of $\Gamma(f)$ in $E\times\R^{m}$ can be computed as
\begin{equation}
\vec{H} =g^{ij}(f )(\bar{\nabla }_{\F{\P }{\P x_i}}\F{\P }{\P x_j}+\F{\P ^2f^A}{\P x_i \P x_j}\F{\P}{\P y^A})^{\bot},
\end{equation}
where $\{\F{\P }{\P y^1},\cdots, \F{\P}{\P y^m}\}$ is the orthonormal frame of $\R^m$ and $( \cdot)^\bot$ denotes the projection onto the normal bundle of $\Gamma(f)$. Set
\begin{equation}
	f_{ij}^A=\F{\P^2f^A}{\P x_i\P x_j}-(\bar{\nabla }_{\F{\P}{\P x_i}}\F{\P}{\P x_j})(f^A),\quad A=1,\cdots, m.
\end{equation}
Therefore, $\Gamma(f)$ is minimal (that is $\vec{H}\equiv 0$) in $E\times \R^m$ if and only if $g^{ij}(f)f_{ij}^A=0$ for each $A=1,\cdots, m$.

 Suppose $\psi=(\psi^{1},\ldots,\psi^{m}):\bar{E}\rightarrow \R^m $ is a smooth map on $\bar{E}=E\cup\partial E$. The Dirichlet problem for the minimal surface system (MSS) over $E$ is defined to seek the solution of
	the following quasi-linear system
	\begin{equation}\label{DP:MSSB}
		\left\{
		\begin{aligned}
			g^{ij}(f)f^A_{ij}&=0\quad A=1,2,\cdots,m\text{\quad on~} E,\\
			f &=\psi \text{\quad on~}\P E,
		\end{aligned}
		\right.
	\end{equation}
	where $(g^{ij}(f))$ is as in \eqref{coe:g}.

	 In the codimension one case ($m=1$), \eqref{DP:MSSB} is simplified into the Dirichlet problem of the minimal surface equation as
	\begin{equation}\label{DP:MSE}
		\left\{
		\begin{aligned}
			 div (\F{D f}{\sqrt{1+|D f|^2}})&=0 ~ \text{\quad on~} E,\\
			f &=\psi \text{\quad on~ }\P E,
		\end{aligned}
		\right.
	\end{equation}
	where $\psi:\bar{E}  \rightarrow \R$ is a continuous map, $f:\bar{E}\rightarrow\R$, and $div$ denotes the divergence of $E$.

	Recall that a smooth domain is \textit{convex} if the second fundamental form of its boundary is non-negative definite, and it is \textit{mean-convex} if the mean curvature of its boundary is non-negative. Some partial results on the Dirichlet problem \eqref{DP:MSE} are recorded as follows.

	\begin{theorem}\label{codimension:one}
	Suppose $E\subset\R^{n}$ is a smooth bounded domain. Then there exists a unique solution $u\in C^\infty(E)\cap C(\bar{E})$ to \eqref{DP:MSE} if one of the following conditions holds:
		\begin{enumerate}
			\item (Jenkins-Serrin \cite{Jenkins-Serrin1968}) the domain $E$ is mean-convex and $\psi$ is continuous;
			\item (Jenkins-Serrin \cite{Jenkins-Serrin1968}) $\psi\in C^2(\bar{E})$ satisfies $w(\psi)<\Sc$ and $\sup_{E}\{|D\psi|+|D^2\psi|\}<K$ for some constant $K>0$ and for some constant $\Sc =\Sc(n,K,E)>0$;
			\item (Williams \cite{Williams1984}) $\psi\in C^1(\bar{E})$ satisfies $ \sup_{\P E}|D\psi|<K$ for any constant $K\in (0, \F{1}{\sqrt{n-1}})$ and $w(\psi)<\Sc$  for some $\Sc=\Sc(n,K,E)>0$.
		\end{enumerate}
    Here $w(\psi)=\sup_{\bar{E} }\psi-\inf_{\bar{E}}\psi$.
 \end{theorem}
In the higher codimensional case ($m\geq 2$), the Dirichlet problem for the MSS \eqref{DP:MSSB} was resolved for both convex and mean-convex domains in Euclidean spaces, respectively, as follows.
\bt\label{codimension:higher} Let $E$ be a smooth compact domain in $\R^n$ with diameter $\eta$ and let $\psi:\bar{E}\rightarrow \R^m$ be a smooth map. Then the Dirichlet problem \eqref{DP:MSSB} for the MSS is solvable if one of the following two conditions holds:
\begin{enumerate}
    \item (Wang \cite{Wang2004}) $8n\eta\sup_{E}|D^2\psi|+\sqrt{2}\sup_{\P E}|D\psi|<1$ and $E$ is convex;
    \item (Ding-Jost-Xin 
    \cite{Ding-Jost-Xin2023}) $\eta\sup_{E}|D^2\psi|+\sup_{E}|D\psi|<\epsilon_{0}$ and $E$ is mean-convex for some constant $\epsilon_{0}=\epsilon_{0}(n,m,\eta,\psi)>0$.
    \end{enumerate}
\et
\begin{remark} \label{remark:motivation} Comparing with  Theorem \ref{codimension:one} and Theorem \ref{codimension:higher}, it is natural to ask the following questions regarding \eqref{DP:MSSB}. Is the mean-convex  assumption essentially indispensable? How should we understand the role of the diameter of $E$ when solving \eqref{DP:MSSB}? These questions are the main motivations of this paper.
\end{remark}
	\subsection{Mean curvature flow}
	 Let $E\subset M$ be a smooth domain given in Section \ref{Sec 2.1}.
For any smooth map $f:E\rightarrow \R^m$, its singular values $\lambda_1,\cdots, \lambda_n$ are the square root of the eigenvalues of the matrix $(df)_{x}^T(df)_{x}$, where $(df)_x:T_x E\rightarrow  \R^m$ is the differential map of $f$ at $x$.
\begin{definition}\label{def:strictly:decreasing}
We say that a smooth or {\Lp} map $f:E\rightarrow\mathbb{R}^{m}$ is (\textit{strictly}) \textit{length decreasing} if the singular values of $df$ satisfy  $\lambda_i\leq 1(<1-\Sc)$ for  $1\leq i\leq n$, and $f$ is \textit{(strictly) area decreasing} if $\lambda_i\lambda_j\leq1(<1-\Sc)$ for  $1\leq i\neq j\leq n$ whenever the differential $df$ is well-defined. Here $\Sc$ is a fixed positive constant in $(0,1)$.
\end{definition}
 It is easy to verify that this definition is independent of the choice of orthogonal frame near $x$ and $f(x)$.

 Next we define the mean curvature flow (MCF) according to \eqref{DP:MSSB} as in \cite{Wang2004}. Suppose $F_{t}(x)=F(x,t)$, where $F:\bar{E} \times [0,T)\rightarrow \bar{E} \times \R^m$, is a smooth parametric solution to the following mean curvature flow
	\begin{equation}\label{MCF}
	\left\{\begin{aligned}
	(	\frac{\partial F  }{\partial t})^{\perp}  &=H; \\
		F |_{\partial E}  &=Id_{E}\times \psi|_{\partial E} .
	\end{aligned}
	\right.
\end{equation}
Here $ Id_{E}$ is the identity map on $E$.
Under a local coordinate system $\{x_{1},\ldots, x_{n}\}$ on $E$, the mean curvature flow $F(x_{1},\ldots, x_{n},t)=( F^{1},\ldots,F^{m})$ satisfies the following parabolic system:
	\begin{align}\label{MCF2}
		\frac{\partial F ^{A}}{\partial t}=\big(g^{ij}\frac{\partial ^2 F^{A}}{\partial x_{i}\partial x_{j}}\big)^{\perp},\  A=1,\ldots,m.
	\end{align}
	Here $(g^{ij})=(g_{ij})^{-1}$ is the inverse of the induced metric $ g_{ij}=\langle \frac{\partial  F}{\partial x_{i} },\frac{\partial  F}{\partial x_{j} }\rangle$, and $(\cdot )^{\perp}$ denotes the projection onto the normal bundle of $F(E,t)$ in $E\times \R^{m}$.

	Suppose the image of $F=(F^{1},\ldots, F^{m})$ is a graph over  $E$. Then there exists a family of diffeomorphisms $r_{t}:E\rightarrow E$ such that $\tilde{F_{t}}=F_{t}\circ r_{t}$ has the form
	\begin{align*}
	\tilde{F}  (x_{1},\ldots, x_{n},t)=(x_{1}, \ldots, x_{n}, f^{1}(x,t),\ldots,f^{m}(x,t)).
	\end{align*}
Here $f_{t}(x)=f(x,t)$, where $ f=(f^{1}(x,t),\ldots,f^{m}(x,t)): E \times [0,T)\rightarrow \mathbb{R}^{m}$, satisfies  non-parametric system
\begin{equation}\label{MCF3}
\left\{\begin{aligned}
\frac{\partial}{\partial t} f^{A} (x,t )&=g^{ij}\frac{\partial^{2}f^{A}}{\partial x_{i}\partial x_{j}}\quad A=1,\ldots,m, ~\mbox{on} ~E\times [0,T)\\
f (x,t)&= \psi ( x),\quad \,\mbox{on} ~\partial E\times [0,T) .
\end{aligned}
\right.
\end{equation}


Now, we collect some facts about the singular values of a smooth map $f:E\rightarrow \R^{m}$. We use similar notations as in \cite{Tsui-Wang2004}. Suppose $r$ is the rank of differential map $df$ at $x\in E$. By the singular value decomposition theorem, there exist an orthonormal basis $\{a_1,\cdots , a_n\}$ in $T_xE$ and an orthonormal basis $\{a_{n+1},\cdots, a_{n+m}\}$ in $T_{f(x)}\R^m$ such that
   \begin{equation}\label{relation:A}
       df(a_i)=\lambda_i a_{n+i}, 1\leq i\leq r,\ \text{and}\  df(a_i)=0,  r+1\leq i\leq n.
              \end{equation}
Using this frame we show Lemma \ref{lm:gij} as follows.
\begin{proof}[The proof of Lemma \ref{lm:gij}] Fix any $(x,f(x))\in E\times\R^m$. Notice that both inequalities in Lemma \ref{lm:gij} are invariant under any coordinate system composed with an orthonormal transformation in $T_x E$ and $T_{f(x)}\R^m$. We may assume the frame $\{a_1,\cdots, a_n\}$ is equal to $\{\F{\P}{\P x_1},\cdots, \F{\P}{\P x_n}\}$ at $x$, and $\{a_{n+1},\cdots, a_{n+m}\}$ is equal to $\{\F{\P}{\P y^{ 1}},\cdots, \F{\P}{\P y^{ m}}\}$ at $f(x)$. Under this setting,
    $$
    g_{ij}=(1+\lambda_i\lambda_j)\delta_{ij},\ \text{and} \  g^{ij}=\frac{1}{1+\lambda_i\lambda_j}\delta_{ij}.
    $$
Notice that $\lambda _i\leq |Df|$ for $i=1,\cdots,n$. Therefore $(1+|Df|^2)(\delta_{ij})\geq (g_{ij})\geq  (\delta_{ij})$ and $(\delta_{ij})\geq (g^{ij})\geq\F{1}{1+|Df|^2} (\delta_{ij})$. The proof is complete.
    \end{proof}
We view $\Gamma(f)$ as a submanifold in the product manifold $E\times \R^m$. Denote the point $(x,f(x))$ by $p$.  Therefore
	\begin{equation}\label{def:tangent:vector}
		e_{i}:=\left\lbrace
		\begin{aligned} 	&\frac{1}{\sqrt{1+\lambda_{i}^{2}}}(a_{i}+\lambda_{i}a_{n+i}),\quad i=1,\cdots,r;\\
			&a_i,\quad\quad\quad i=r+1,\ldots,n,
		\end{aligned}\right.
	\end{equation}
forms an orthonormal basis of the tangent space $T_{p}\Gamma(f)$,  and
	\begin{equation}\label{def:normal:vector}
		e_{n+i}:=\left\lbrace
		\begin{aligned}
            &\frac{1}{\sqrt{1+\lambda_{i}^{2}}}(-\lambda_{i}a_{i}+a_{m+i})\quad i= 1,\ldots,r;\\
			&a_{m+i},\quad\quad\quad i=r+1,\cdots, n,
		\end{aligned}\right.
	\end{equation}
forms an orthonormal basis of the normal space $N_{p}\Gamma(f)$.

From now on, we assume that the Ricci curvature of the Riemannian manifold $M$ is non-negative, that is $Ric_M\geq 0$. Our purpose is to show Proposition \ref{prop:preserved}, which says that the strictly length decreasing property is preserved along the MCF in $E\times \R$ with fixed boundary.

With the above notation, let $\Omega$ be the volume form of $E$ such that
	$\Omega(a_1,\cdots, a_n) =1$,
and so $*\Omega$, the projection of $\Omega$ onto $\Gamma(f)$, can be written as
	\begin{equation}\label{star Omega}
		*\Omega=\Omega(e_1,\cdots, e_n)=\F{1}{\sqrt{\prod_{i=1}^n(1+\lambda_i^2)}},
	\end{equation}
which is independent of the choice of local orientable frame on $\Gamma(f)$.

For the graphic mean curvature flow $\Sigma_t=F(x,t)=\{(x,f_t(x)):x\in E\}$ defined in \eqref{MCF2}, Wang \cite{Wang2002,Wang2003,Wang2004}  and Tsui-Wang \cite{Tsui-Wang2004}  have already done a lot of calculation. We record the evolution equation of $*\Omega$  for completeness .

In the followings, let $\bar{R}$ and $R$ denote the curvature tensors of $E\times\R^m$ and $E$, respectively. Let the index $i,j,k,l$ range from $1$ to $n$, and let $\alpha$ range from $n+1$ to $n+m$. Suppose the MCF flow $\Sigma_t$ exists on $[0,T)$ in $E\times \R^m$. By noting $(\ln *\Omega)_k= -\sum_{i} \lambda_i h_{n+i,ik}$, for each $k $,  we have
   \begin{equation}\label{evolution Omega}
   	\begin{aligned}
   		&(\F{d}{dt}-\Delta)\ln *\Omega = \sum_{\alpha ,l,k} h_{\alpha  lk}^{2}+2\sum _{k,i<j}\lambda_{i}\lambda_{j}   h_{n+i,jk}h_{n+j,ik}
           \\
     &+\sum_{i,k} \lambda_{i}^{2} h_{n+i,ik}^{2}-\F{1}{*\Omega}\sum_{\alpha,k}(\Omega_{\alpha 2\cdots n}\bar{R}_{\alpha kk1}+\cdots+\Omega_{1\cdots ( n-1)\alpha }\bar{R}_{\alpha kk n})
   	\end{aligned}
   \end{equation}
   where $h_{\alpha,ij}=\la \bar{\nabla}_{\tau_i}\tau_j, \tau_{\alpha} \ra$, $\bar{R}_{\alpha kkj}=\bar{R}(e_{\alpha},e_k,e_k,e_j)$, and $\Omega_{1\cdots (k-1)\alpha (k+1)\cdots n}=\Omega(e_1,\cdots, e_{k-1},e_\alpha, e_{k},\cdots,e_n)$. Putting \eqref{def:tangent:vector} and \eqref{def:normal:vector}  into the curvature term in   \eqref{evolution Omega}, we obtain
   	\begin{align*}
   		&-\frac{1}{*\Omega}\sum_{\alpha,k}(\Omega_{\alpha 2\cdots n}\bar{R}_{\alpha kk1}+\cdots+\Omega_{1\cdots n-1\alpha }\bar{R}_{\alpha kk n})\\
   		=&\sum_{i}  \frac{\lambda_i^2}{1+\lambda_i^2}\prod_{j\neq i}\frac{1}{1+\lambda_j^2}\sum_{k}R(a_i, a_k,a_j, a_k) \\
   		=& \sum_{i}  \frac{\lambda_i^2}{1+\lambda_i^2}\prod_{j\neq i}\frac{1}{1+\lambda_j^2} Ric_M(a_i, a_j)\geq 0.
   	\end{align*}
Since $Ric_{M}\geq 0$, we obtain
\begin{equation}\label{eq star Omega}
               (\F{d}{dt}-\Delta)\ln *\Omega\geq  \sum_{\alpha ,l,k} h_{\alpha, lk}^{2}+2\sum _{k,i<j}\lambda_{i}\lambda_{j}   h_{n+i,jk}h_{n+j,ik}+\sum_{i,k} \lambda_{i}^{2} h_{n+i,ik}^{2}.
\end{equation}
If $f_t$ preserves the strictly area decreasing property, $\lambda_i\lambda_j<1-\Sc$ for some positive constant $\Sc$ independent of $t$, as \cite[(6.2)]{Tsui-Wang2004}, the right hand side of \eqref{eq star Omega} is greater than $\Sc \sum_{\alpha,l,k}h_{\alpha,lk}^2$. Therefore it holds that
\begin{equation} \label{del:min:A}
     (\F{d}{dt}-\Delta)\ln *\Omega\geq \Sc\sum_{\alpha ,l,k} h_{\alpha, lk}^{2}
    \end{equation}
     \begin{remark}
        According to the maximum principle, the condition $*\Omega>0$ is not broken along the interior of the graphic MCF flows. \label{thegraph:remark}  As a result, under the assumption of Proposition \ref{prop:preserved} below, if the strictly length decreasing property of $f_t(x)$ is preserved along the MCF $\Sigma_t$, it remains a graph.
    \end{remark}
     To study the length decreasing property, consider the symmetric 2-tensor $S$ defined by
        \begin{equation*}
            S(X,Y)=\la  \pi_1 (X), \pi_1 (Y)\ra-\la  \pi_2 (X),  \pi_2 (Y)\ra
            \end{equation*}
           for any $X,Y\in T(E\times \mathbb{R}^{m})$, where $\pi_1:E\times\R^m\rightarrow E$ and $\pi_2:E\times \R^m\rightarrow \R^m$ are the corresponding projections. From \eqref{def:tangent:vector},
            \begin{equation}\label{S_ij}
               S_{ij}=S(e_i, e_j)=\frac{1-\lambda^2_{i}}{1+\lambda_i^2}\delta_{ij}.
                \end{equation}
      Thus $S$ is positive definite restricted on $T\Sigma_{t}$ if and only if $\lambda_i <1$, for $i=1,\cdots, n$  (that is $f_{t}$ is strictly length decreasing).  For future use, we point out that
      \begin{equation}\label{dust} S(e_{i}, e_{n+j})=S(e_{n+i}, e_j)=\frac{ -2\lambda _{i}}{1+\lambda_i^2}\delta_{ij} \quad  S(e_{n+i}, e_{n+j})=-S_{ij},
          \end{equation}  by \eqref{def:tangent:vector} and \eqref{def:normal:vector}.

The following theorem is an analog of the strictly length decreasing preserving result in \cite[Section  4]{Tsui-Wang2004} in the boundary case.
     \begin{prop} \label{prop:preserved}
     	Let $M^{n}$ be a complete Riemannian manifold with non-negative Ricci curvature and let $E\subset M$ be a bounded smooth domain. Suppose $\psi:\bar{E}\rightarrow \R^m$ is a smooth strictly length decreasing map, and the mean curvature flow $\Sigma_t$ in \eqref{MCF3} exists on $\bar{E} \times [0, T)$ as a smooth graph $(x,f_t(x))$ satisfying $\max\{\sup_{ \P E\times  [0,T)}|Df_t|,\sup_{ \bar{E} }|D\psi|\}<1-\Sc $, for a fixed constant $\Sc \in (0,1)$. Then $\Sigma_t$ stays a graph of a smooth strictly length decreasing map  on $\bar{E} \times [0, T)$ with $\sup_{E}|Df_t|< 1-\Sc$.
     \end{prop}
     \begin{proof}
     By Remark \ref{thegraph:remark},  
     it suffices to show that $ \sup_{E}|D f_{t}|<1-\Sc$ for any $t\in [0,T)$.


     By the assumption, define $P_{ij}=S_{ij}-\epsilon' g_{ij}$ along the mean curvature flow $\Sigma_t$. Here $\Sc'=\frac{1-(1-\epsilon)^{2} }{1+(1-\epsilon)^{2}}\in (0,1)$ is a constant such 
     	 that $P_{ij}>0$ (positive definite) on $\P E\times [0,T)\cup E\times \{ 0\} $ according to \eqref{S_ij}. It remains to show that $P_{ij}>0$ is preserved along the interior of $\Sigma_t$ over the interval $[0,T)$.

     A straightforward computation gives that (c.f. \cite[ Eq. (4.8)]{Tsui-Wang2004})
     	\begin{align*}
     		(\frac{d}{dt}-\Delta )P_{ij}
     		=& -h_{\alpha il}H_{\alpha}P_{lj}-h_{\alpha jl}H_{\alpha}P_{li}+\bar{R}_{kik\alpha}S_{\alpha j}+\bar{R}_{kjk\alpha}S_{\alpha i}
     		\\&+h_{\alpha kl}h_{\alpha ki} P_{lj} +h_{\alpha kl}h_{\alpha kj} P_{li}+2\epsilon' h_{\alpha ki}h_{\alpha kj}-2h_{\alpha ki}h_{\beta kj} S_{\alpha\beta},
     	\end{align*}
          on $E\times [0, T)$.
         Let $Q_{ij}$ denote the right hand-side of the above equation.
         By the Hamilton's maximum principle for tensors (see \cite[Theorem 9.1]{Hamilton1982} or \cite{Hamilton1986}), if $Q_{ij}w^{i}w^{j}\geq 0$ for any vector $w=(w^1,\cdots, w^{n})^T$ on $E\times [0, T)$  such that  $P_{ij}w^{j}= 0$ for all $1\leq i\leq n$, then $P_{ij}$ is positive definite on $E\times[0,T)$.

         Indeed, we compute
     	\begin{align} \label{Q positive}
     		Q_{ij}  w^{i}w^{j}
     		=&  2\bar{R}_{kik\alpha}S_{\alpha j}w^{i}w^{j}
     		+2\epsilon' h_{\alpha ki}h_{\alpha kj}w^{i}w^{j}-2h_{\alpha ki}h_{\beta kj} S_{\alpha\beta}w^{i}w^{j}\nonumber
     		\\\geq& 2\bar{R}_{kik\alpha}S_{\alpha j}w^{i}w^{j} -2h_{\alpha ki}h_{\beta kj} S_{\alpha\beta}w^{i}w^{j}.
     	\end{align}
     	For the second term in \eqref{Q positive}, using $S_{ij}\geq \epsilon g_{ij}\geq0$ and the relation in \eqref{dust}, we have
     	\begin{align*}
     		&-2h_{\alpha ki}h_{\beta kj} S_{\alpha\beta}w^{i}w^{j}
     		\\=&2\sum_{1\leq p,q\leq r} h_{n+p, ki}h_{n+q, kj} S_{pq}w^{i}w^{j}+2\sum_{r+1\leq p,q\leq m} h_{n+p, ki}h_{n+q, kj}  w^{i}w^{j}\geq 0.
     	\end{align*}
     For the first term in \eqref{Q positive}, we  prove that $\bar{R}_{kik\alpha}S_{\alpha j} \geq 0$.
     	Since
     	\begin{align*}
     		\sum_{k} \bar{R} (e_{ \alpha},e_{ k},e_{k},e_{i})
     		= &\sum_{k}R(\pi_{1}(e_{ \alpha}),\pi_{1}(e_{ k}),\pi_{1}(e_{k}),\pi_{1}(e_{i}))
     		\\= & \sum_{k}R( e_{ \alpha},e_{ k},e_{k},e_{i})=Ric_{M} ( e_{ \alpha}, e_{i}).
     	\end{align*}
Here we used the definition of sectional curvature  $K(X,Y)=R(X,Y,Y,X)$.
     	As a result, by \eqref{dust}, we have
     	\begin{align*}
     		\sum_{\alpha,k} \bar{R}_{kik\alpha}S_{\alpha j}
     		=  &- \sum_{p,k} \bar{R}_{n+p, kki }S_{n+p, j} =-\sum_{p }Ric_{M}( e_{n+p}, e_{i})S_{n+p, j}
     		\\=& Ric_{M}( e_{n+j}, e_{i}) \frac{ 2\lambda_{j}}{1+\lambda_{j}^{2}}  \geq0.
     	\end{align*}
   In summary, plugging the above computations into \eqref{Q positive} gives $Q_{ij}v^i v^j\geq 0$. This completes the proof.
     \end{proof}

\section{Boundary gradient estimates}\label{Sec bdry est}

	In this section, we derive the boundary gradient estimate along the mean curvature flow, which is essential in the proof of Theorem \ref{main:thm:A}.

	Denote the graphic mean curvature flow by  $\Sigma_{t }=(x,f_{t}(x))$, where $f_t(x):\bar{E}\times [0,T)\rightarrow \R^m$. Let $df_t$ denote the differential of $f_t$. Define $d(x):=dist(x,\partial E)$, for $x\in E$ as the distance from $x$ to the boundary $\partial E$.

	\begin{lem}\label{lap}
	Let $M^{n}$ be a complete Riemannian manifold and let $E\subset M$ be a bounded smooth domain. Suppose the mean curvature flow $\Sigma_{t}$ exists on $E\times[0,T)$. Then there exists a constant $\eta_0=\eta_0(n, \P E)>0$, and $c_{0}=c_{0}(n,\partial E )>0$,  such that
		\begin{align*}
			-\Delta_{\Sigma_{t}} d \geq -c_{0}, \  \   \mbox{on} \ E_{\eta_0} ,
		\end{align*}
	 where $E_{\eta_{0}} =\{x\in E:d(x,\partial E)<\eta_{0}\}$.  Additionally, if $E$ is strictly convex,  $c_0$ can be chosen as $0$.
	\end{lem}
    \begin{remark}
    	The main role of $\eta_0$ is to ensure that the function $d$ is a $C^2$ function on $E_{\eta_0}$. For instance, whenever $E$ is strictly convex, we shall choose $\eta_0$ sufficiently small such that the level set of $d(x)$ in $E_{\eta_0}$ is embedded and convex with respect to the outward normal vector.
        \end{remark}
	\begin{proof}
	It is easy to see that $d(x)$ is a $C^2$ function on $E_{\eta_{0}}$, for some  $\eta_{0}=\eta_{0}(n,\partial E)>0$.
Fix any $x\in E_{\eta_{0}}$, we use the induced metric of $\Sigma_t$ given in \eqref{coe:g}. Let $(\sigma_{ij})$ be the metric matrix of $E$. By Lemma \ref{lm:gij},$(g_{ij})\leq (\sigma_{ij})$. For any $x\in E_\eta$, choose an orthonormal frame $\{a_1,\cdots,a_n\}$ at $T_x E$. Without loss of generality, we assume $(\sigma_{ij})$ is the identity matrix at $x$. Hence $(g_{ij})\leq (\delta_{ij})$. Then at any $x\in E_{\delta_0}$,
		\begin{align}
			-\Delta_{\Sigma_{t}} d =&-g^{ij}\mathrm{Hess}(d) (a_{i}, a_{j})
			\geq
		  -  \sum_{i} |\mathrm{Hess}(d) (a_{i}, a_{i})|\label{eq:computation:d}\\
          \geq& -n\max_{i}\sup_{E_{\eta_0}}|\mathrm{Hess}(d) (a_{i}, a_{i})|:=-c_{0}(n,\partial E).\notag
		\end{align}
     \noindent   In the case that $ E$ is strictly convex, the matrix $\{-\mathrm {Hess}(d)(a_i,a_j)\}$ is positive definite on $\P E$. As a result, if we take $\eta_0$ sufficiently small, it holds that
        $-g^{ij}\mathrm{Hess}(d)(a_i,a_j)\geq 0$ on $E_{\eta_0}$. In the above inequality, $c_0$ could be chosen as $0$.
	\end{proof}
	\begin{remark}\label{delta0rmk}
			In particular, when  $E$ is a ball $B_{r}(0)$ in $\R^n$ centered
            at the origin with radius $ r$, a direct computation yields that  $-\mathrm {Hess}(d)(e_i,e_j)(x)=\F{1}{r-d(x)}\delta_{ij}> 0$ for any $x\in B_{r}(0)$. Then, according to \eqref{eq:computation:d},
			\begin{equation}
				-\Delta_{\Sigma_t}d \geq 0
			\end{equation}
			holds on $\{x\in B_{r}(0): dist(x,\P B_{r}(0))<1\}$.
			Therefore, we can choose $c_0$ as zero and $\eta_0$ as any number in $(0, r)$.
	\end{remark}

To formulate the desirable boundary gradient estimates, we begin by defining some notations to govern the subsequent discussion. Denote $w^{A} =\sup_{\bar{E}}\psi ^{A}-\inf_{\bar{E}}\psi^{A} $ and  $w(\psi) =\max_{1\leq A\leq m}w^{A}$. Let  $f=(f^{A})_{A=1}^{m}$ and $\psi=(\psi^{A})_{A=1}^{m}$. By the maximum principle, $\inf_{\bar{E}}\psi^{A} \leq  f^{A} \leq\sup_{\bar{E}}\psi^{A}$  on $\bar{E}\times[0,T)$ for $1\leq 	A \leq m$.

	\begin{prop}\label{keyLboundary:estimate}
		Let $M^{n}$ be a complete Riemannian manifold and let $E\subset M$ be a bounded smooth domain. Suppose $F=(x,f(x,t)): E\times[0,T)\rightarrow E`	\times \R^{m}$ is a smooth solution to \eqref{MCF3}. Then there exists a constant $ \delta_0=\delta_0(n,\partial E,\mu)>0$  such that for any $\delta\in (0,\delta_0]$,
		\begin{align*}
			|Df | \leq &    \frac{ \omega(\psi)}{\delta} +  |D\psi |+16n(1+\mu) \delta   |D^{2} \psi |, \quad \mbox{on}\quad  ~\partial E \times [0,T),
		\end{align*}
		where $\mu=\sup_{E \times [0,T)}|Df |^{2} $.
	\begin{remark}\label{remark:definitioneta}
	Let $c_0$ and $\eta_0$ be the numbers given in Lemma \ref{lap}. In \eqref{definition:eta_0},  the constant $\delta_0$ is given by
	\begin{align*}
	\delta_0=\begin{cases} \frac{1}{2}\min \{\frac{1}{8c_{0}(1+\mu) }, \eta_{0}\}, \quad \mbox{if}\  c_0>0;\\
		\eta_0, \quad\mbox{if}\  c_0=0 \quad( \text{when $E$ is strictly convex}).
	\end{cases}
	\end{align*}
    This Proposition generalizes the computation in \cite[Proposition 2.2]{Wang2004}.
	\end{remark}
	\end{prop}
	\begin{proof}
		Choose two coordinate systems  $\{x^{i}\}$ on $E$  and   $\{y^{A}\}$ on $\mathbb{R}^{ m}$, respectively. Then
            $F(x,t)=(x^1,\cdots, x^n, f^1(x,t),\cdots, f^m(x,t))$ and  $ \psi(x)=(\psi^1(x),\cdots,\psi^m(x)) $. For each  $A=1,\ldots,m$, define
		\begin{align*}
			S (x,t)=\nu \log (1+kd(x))+\psi ^{A}-f ^{A}+\frac{\omega^{A}}{\delta}d(x), \quad\mbox{on}~ E_{\delta}\times[0,T),
	  \end{align*}
		where $k,\nu$  and $\delta$ are positive constants to be determined later.

		Firstly, we claim that $S\geq 0$ on $\partial E_{\delta}\times[0,T)$. In fact, note that $\partial E_{\delta}=\partial E \cup\partial^{\delta}E$, where $\partial^{\delta} E:=\{x\in E:d(x,\partial E)=\delta\}$.  Obviously, $S=0$ on $\partial E\times[0,T)$. By the definition of $\omega^A$,  $\omega^{A}=\sup_{\bar{E}}\psi^A-\inf_{\bar{E}}\psi^A$, it holds that
		\begin{align*}
			S  \geq   
			\psi^{A}-f^{A}+\sup_{\overline{E}}\psi^{A}-\inf_{\overline{E}}\psi^{A}\geq 	\sup_{\overline{E}}\psi^{A}- f^{A} \geq 0,  \quad\mbox{on}~ \partial^{\delta}E\times[0,T).
		\end{align*}
		On the other hand, since $f\equiv \psi$ on $\P E\times [0,T)$, we see that $S\geq 0$ on $ \P E\times [0,T)$. These two facts together yield the claim.

		Next, a direct calculation gives
		\begin{align}\label{dt-laplace}
			&(\dfrac{d}{dt}-\Delta _{\Sigma_{t}})S \\
			=&-\bigg(\frac{ \nu k}{1+kd (x)}+ \frac{\omega^{A}}{\delta} \bigg)\Delta _{\Sigma_{t}}d(x) +\frac{\nu k^{2}}{(1+kd(x))^2}g^{ij}\frac{\partial d }{\partial x^{i} }\frac{\partial d }{\partial x^{j} }-\Delta_{\Sigma_{t}} \psi ^{A} \nonumber.
		\end{align}

		Let $\mu=\sup_{E \times [0,T)}|Df |^{2} $. By Lemma \ref{lm:gij} $\frac{1}{1+\mu} (\sigma^{ij})\leq (g^{ij})\leq (\sigma^{ij})$. Applying Lemma \ref{lap}, we derive
		\begin{align*}
			(\dfrac{d}{dt}-\Delta _{\Sigma_{t}})S \geq-c_{0}\bigg(\dfrac{ \nu k}{1+kd (x)}+ \frac{\omega^{A}}{\delta} \bigg)+\frac{\nu k^{2}}{(1+kd (x))^2} \frac{1}{1+\mu} -|\Delta_{\Sigma_{t}} \psi ^{A}|,
		\end{align*}
		where $c_{0}=c_{0}(n,\partial E)>0$ is the constant given in Lemma \ref{lap}.

		Further,  using $|\Delta_{\Sigma_{t}} \psi ^{A}|  \leq n|D^{2}\psi^{A}|$
		and $d(x)< \delta$ on $E_{\delta}$, we obtain
		\begin{align*}
			(\frac{d}{dt}-\Delta _{\Sigma_{t}})S
			\geq&-c_{0}( \nu k +\frac{ \omega^{A}}{\delta})+ \frac{\nu k^{2}}{(1+k\delta)^2}\frac{1}{1+\mu} -n|D^{2} \psi ^{A}|.
		\end{align*}
		Choose $k= \delta^{-1}$. We derive
		\begin{align*}
			(\frac{d}{dt}-\Delta _{\Sigma_{t}})S
			\geq&-c_{0}\frac{( \nu + \omega^{A})}{\delta}+  \dfrac{\nu}{4(1+\mu)\delta^{2}} -n|D^{2} \psi ^{A}|\\
			=&\frac{\nu  }{\delta}\big (   \frac{1}{4(1+\mu) \delta} -c_{0}  \big)-c_{0}\frac{  \omega^{A} }{\delta}-n|D^{2} \psi ^{A}|.
		\end{align*}

	It suffices to take $\nu$ as
		\begin{align*}
			\nu =& \frac{4(1+\mu) \delta ^{2} }{1-4c_{0}(1+\mu) \delta}\big(c_{0}\frac{  \omega^{A} }{\delta}+n|D^{2} \psi ^{A}|\big)>0.
		\end{align*}
		Then
		\begin{align*}
			(\frac{d}{dt}-\Delta _{\Sigma_{t}})S\geq 0 \quad \text{ on} \quad E_\delta\times [0,T)
		\end{align*}

		To determine $\delta$, our consideration is divided into two cases.

		\textit{Case 1}.
		Suppose $c_{0}=0$. Now
		\begin{align} \label{c0=0}
			\dfrac{\nu }{\delta} =&  4(1+\mu) \delta n|D^{2} \psi ^{A}| .
		\end{align}
	Let $\delta_0=\eta_0$, where $\eta_0$ is   given as in Lemma \ref{lap}. We can choose $\delta$ to be any number in $(0,\delta_0]$.

	\textit{Case 2}. Suppose $c_{0}>0$.	Take $\alpha_0=\min \{\frac{1}{8c_{0}(1+\mu) }, \eta_{0}\}$. For any $\delta\in (0,\alpha_0]$, then  $4c_{0}(1+\mu)\delta\leq 1/2$ and so
		$1-4c_{0}(1+\mu) \delta \geq1/2$.
		Thus
		\begin{align}  \label{nu-delta-mu}
			\frac{\nu }{\delta} =& \frac{4(1+\mu) \delta  }{1-4c_{0}(1+\mu) \delta}\big(c_{0}\frac{  \omega^{A} }{\delta}+n|D^{2} \psi ^{A}|\big)\\\nonumber
			\leq&  8(1+\mu) \delta   \big(c_{0}\frac{  \omega^{A} }{\delta}+n|D^{2} \psi ^{A}|\big)\\\nonumber
			=&  8(1+\mu) c_{0}   \omega^{A} +8(1+\mu) \delta   n|D^{2} \psi ^{A}|\\\nonumber
			\leq &    \frac{ \omega^{A}}{\delta} +8n(1+\mu) \delta   |D^{2} \psi ^{A}| ,
		\end{align}

		In both cases, we arrive at $(\frac{d}{dt}-\Delta _{\Sigma_{t}})S\geq0$.
		Applying the maximum principle, we have $S \geq 0$ in $E_{\delta}\times[0,T)$. Namely $f^{A}- \psi ^{A}\leq \nu \log (1+kd(x))+ \frac{\omega^{A}}{\delta}d(x)$  on $E_{\delta}\times[0,T)$.

		Similarly, applying the above process to the function
		$\tilde{S} =\nu\log(1+kd (x ))+f^{A}-\psi^{A}+ \frac{\omega^{A}}{\delta}d(x)$, we obtain $ \psi ^{A}-f^{A}\leq \nu \log (1+kd(x))+ \frac{\omega^{A}}{\delta}d(x)$  on $E_{\delta}\times[0,T)$.

		Therefore, we obtain  
		\begin{align*}
			\left| \dfrac{\partial(f^{A}-\psi^{A}) }{\partial \mathbf{n}} \right|
			\leq&\lim_{d (x)\rightarrow0}	\dfrac {|f^{A}(x)-\psi^{A}(x)|}{d (x)}\\
			\leq &\lim_{d (x)\rightarrow0}\dfrac{\nu\log(1+kd (x))}{d (x)}+ \frac{\omega^{A}}{\delta}\\
			\leq&\nu k+\frac{ \omega^{A} }{\delta} = \frac{\nu +\omega ^{A}}{\delta}
		\end{align*}
		where $ \mathbf{n}$ is the outward unit  normal vector of $\partial E$.
		Assuming $\frac{\partial f^{A}}{\partial \mathbf{n}}=0$ for all $A\geq2 $ up to an orthogonal transformation of coordinates on $E$ at a fixed point, we derive
		\begin{align*}
			\left|\frac{\partial f}{\partial \mathbf{n}} \right|  \leq \frac{\nu +w(\psi) }{\delta}+\left| \dfrac{\partial \psi}{\partial \mathbf{n}}\right| .
		\end{align*}

		On the other hand, we have 
		\begin{align*}
			|D^{\top}f |= |D^{\top}\psi | \quad \mbox{on} ~ \partial E,
		\end{align*}
		where  $|D^{\top}f |(x)=\sup_{v\in T_{x}(\partial E),|v|=1}|Df(x)(v)|$.

		Therefore,
		\begin{align} \label{nu+omega}
			|Df |\leq &\sqrt{\big( \frac{\nu +w(\psi) }{\delta}+ | \frac{\partial \psi}{\partial \mathbf{n}}|\big )^{2}+|D^{\top}\psi |^{2}}\\
			\leq &\frac{\nu +w(\psi) }{\delta}+ |D\psi | \quad \mbox{on}~ \partial E.\nonumber
		\end{align}
		Here, in the last inequality, we use
		\begin{align*}
			&\big( \frac{\nu +w(\psi) }{\delta}\big)^{2}+| \frac{\partial \psi}{\partial \mathbf{n}} |^{2}+2\big( \frac{\nu +w(\psi) }{\delta}\big) | \frac{\partial \psi}{\partial \mathbf{n}} |+|D^{\top}\psi |^{2}\\
			\leq &(\frac{\nu +w(\psi) }{\delta})^{2}+  |D\psi | ^{2}+2 |D\psi | \frac{\nu +w(\psi) }{\delta}=(\frac{\nu +w(\psi) }{\delta}+ |D\psi |)^{2}
		\end{align*}

		Finally, we arrive the following conclusion.

		\textit{Case 1}.
		Combining \eqref{c0=0} and \eqref{nu+omega} together, we obtain
		\begin{align*}
			|Df | \leq &    \frac{ w(\psi) }{\delta} +  |D\psi |+4n(1+\mu) \delta   |D^{2} \psi |,\quad \mbox{on} ~\partial E \times [0,T).
		\end{align*}
      for any $\delta \in (0, \delta_0]$. \\
		\indent \textit{Case 2}.
		Combining \eqref{nu-delta-mu} and \eqref{nu+omega} together, we have
		\begin{align*}
			|Df | \leq &    \frac{ 2 w(\psi) }{\delta} +  |D\psi |+8n(1+\mu) \delta   |D^{2} \psi |,\quad \mbox{on} ~\partial E \times [0,T).
		\end{align*}
        for any $\delta\in (0,\alpha_0]$. Now let $\delta_0=\frac{1}{2}\alpha_0=\frac{1}{2}\min \{\frac{1}{8c_{0}(1+\mu) }, \eta_{0}\}$. The above inequality becomes that
		\begin{align*}
			|Df | \leq &    \frac{ w(\psi)}{\delta} +  |D\psi |+16n(1+\mu) \delta   |D^{2} \psi |.
		\end{align*}
        for any $\delta\in (0,\delta_0]$.

         Combining these cases together, there exists  $\delta_0=\delta_0(n,\mu, \P E)>0$ given by
        \begin{equation}\label{definition:eta_0}
                \delta_0=\begin{cases} \frac{1}{2}\min \{\frac{1}{8c_{0}(1+\mu) }, \eta_{0}\}, \quad c_0>0;\\
                    \eta_0, \quad c_0=0,
                \end{cases}
            \end{equation}
            such that for any $\delta\in (0,\delta_0]$,
   	\begin{align*}
   		|Df | \leq &    \frac{ w(\psi)}{\delta} +  |D\psi |+16n(1+\mu) \delta |D^{2} \psi |.
   	\end{align*}
   on $\P E\times [0,T]$. The proof is complete.
	\end{proof}

	\section{Bernstein-type results over half-spaces}\label{Sec Bernstein}
 In this section, we derive some Bernstein-type results for minimal graphs and self-shrinkers
 over half-spaces.  These results  are essential for establishing the long-time existence of the MCF flow with boundary in Lemma \ref{second:step:MCF}, albeit they may not be optimal.
 \subsection{Smooth Bernstein type theorem} Let $\R^n_{+}$ be the upper half-space $\{(x_1,\cdots, x_n)\in\R^n:x_n>0\}$. Throughout this subsection, $\Omega$ denotes $dx^1\wedge \cdots \wedge dx^n$ which is the volume form of $\R^n$.
	\begin{theorem}\label{smooth:bernstein:theorem}
		Let $f:\bar{E}\subset \R^n \rightarrow \R^m$ be a smooth map on $E$ and Liptschtiz on $\bar{E}$. Suppose $f$ is strictly length decreasing i.e. $\lambda_i <1-\Sc$ for  $1\leq i\leq n$ and for some $\Sc\in (0,1)$. Then the following results hold.

		(1) If $\Gamma(f)$, the graph of $f$, is minimal, then
		\begin{equation}\label{star:Omega}
			0\geq \Delta \ln *\Omega + \Sc \sum_{\alpha,l,k}h^2_{\alpha,lk} \quad \textit{on}\  \Gamma(f).
		\end{equation}

				(2) In addition, if $E=\R^n_{+}$ and $f|_{\P \R^n_{+}}=u_{\P \R^n_{+}}$, then $f$ is a linear map. Here $u_{\P \R^n_+}$ is a linear  map from $\P \R^n_+$ to itself. 
	\end{theorem}
	\begin{proof}
		The conclusion (1) follows from \eqref{del:min:A} and the fact that a minimal submanifold is stationary for the MCF.

	From now on, we assume that $E$ is the half space $\R^n_{+}$. Denote $B=\{(x_1,\cdots,x_{n-1},0, \cdots,0)\}$ which is the subspace of $ \R^{n+m}$. Let $\theta,\theta^{\bot}$ be the orthogonal projections from $\R^{n+m}$ onto $B$ and $B^{\bot}$ (which is the normal space of $B$), respectively. For $z=(x_1,\cdots,x_n,y^1,\cdots,y^m)\in \R^{n+m}$, define $$
		\tau(z):=\theta(z)-\theta^{\bot}(z)=(x_1,\cdots,x_{n-1}, -x_n, -y^1,\cdots, -y^m).$$
		Let $v(\Gamma(f))$ note the integral varifold generated from $\Gamma(f)$. Then it is stationary on $\R^{n+m}\setminus B $. According to the reflection principle in \cite[Section 3.2]{Allard1975}, $v(\Gamma(f))+\tau (v(\Gamma(f)))$ is also a stationary integral varifold which is the graph of some {\Lp} function $\tilde{f}:\R^n\rightarrow \R^m$. Moreover, $\lambda_i <1-\Sc$ for any $i=1,\cdots,n$ whenever $d\tilde{f} $ is well-defined. By \cite[Theorem 4.1]{Wang2004} or Theorem \ref{smooth:c-minimal} below, $\tilde{f}$ is smooth. According to \cite[Theorem 1.1]{Wang2003}, $\tilde{f}$ is a linear map on $\R^n$.
	\end{proof}
  Recall that a smooth submanifold in Euclidean space is called a \textit{self-shrinker} if its mean curvature vector $\vec{H}$ satisfies
\begin{equation}
    \vec{H}+\F{\vec{F}^{\bot}}{2}=0,
\end{equation}
where $\vec{F}$ is the position vector and $\vec{F}^{\bot}$ denotes the projection of $\vec{F}$ into its normal bundle. First, we recall the following two results for self-shrinkers from the thesis of the second author \cite{Zhou15}.
\bt(\cite[Theorem 4.1.2]{Zhou15})\label{thm:st:eq:ad} Let $E$ be an open set in $\R^n$. Suppose $f:E\rightarrow \R^m$ is a smooth map, and $\Gamma(f)$ is a smooth self-shrinker. Then  $*\Omega$ (given by \eqref{star Omega})satisfies
\begin{align}
	&\Delta \ast\Omega -\F{1}{2}\la\vec{F},\ast \Omega\ra-\F{|\nb \ast\Omega|^2}{\ast\Omega}\notag\\ &+\bigg(\sum_{i,k}\lambda^2_i  h_{ n+i,ik} ^2 +\sum_{i,j,\alpha} h_{\alpha,ij}  ^{2} + 2\sum_{i<j,k}\lambda_i\lambda_j h_{n+j,ik}h_{n+i,jk} \bigg)\ast\Omega=0.\label{eq:st_ad}
\end{align}
Here $  h_{\alpha,ij}=\la \bnb_{e_i}e_j, e_\alpha\ra$, $\bnb(\nb)$ is the covariant derivative of $\R^{n+m} (\Gamma(f))$, $\Delta$ is the Laplacian of $\Gamma(f)$ and $\vec{F}$ is the position vector of $\Gamma(f)$.
\et
\begin{remark} The proof of the above structure equations is the same as that in \cite{Wang2002} if we replace the vector $\P_t$ by $-\F{\vec{F}}{2}$.
    \end{remark}
\bt(\cite[Theorem 2.2.1]{Zhou15})\label{thm:v_normal} Let $U$ be an open set in $\R^{n+m}$. If $\Sigma$ is a smooth, properly embedded $n$-dimensional submanifold in $U$ such that its point $\vec{F}$ satisfies
$\vec{F}^{\bot}\equiv 0$
, then $\Sigma$ is totally geodesic in $U$.
\et
\br An advantage of our proof is that it works locally. A proof that a smooth, minimal, complete self-shrinker $\Sigma$ with codimension one is totally geodesic can be briefly stated as follows (See Corollary 2.8, \cite{CM12}). If $\Sigma$ is minimal, then $\sqrt{-t}\Sigma=\Sigma$ since $\sqrt{-t}\Sigma$ is a solution of the mean curvature flow. Therefore $\Sigma$ is a smooth minimal cone, the rigidity of $\Sigma$ follows from that a smooth minimal cone is a totally geodesic plane. If a self-shrinker is minimal only in its open subset, then the above derivations are invalid. In this case, we can not claim the self-shrinker is a smooth minimal cone, which is a global property.
\er
\bp Choose any point $\vec{F}_1$ on $\Sigma$. We denote by $\{e_i\}_{i=1}^n$ the orthonormal frame of the tangent bundle and denote by $\{n^\alpha\}_{\alpha=1}^k$ the normal bundle of $\Sigma$ in a neighborhood of $\vec{F}_1$. Without confusion, we assume this neighborhood is still $U\cap \Sigma$. We define a set $V\subset U\cap \Sigma$ as follows:
$$
V=\{\vec{F}:\vec{F}\in U\cap \Sigma,\quad \la\vec{F}, e_i\ra\neq 0\quad \text{for}\quad i=1,\cdots,n\}
$$
We claim that $V$ is an open dense set in $U\cap \Sigma$.\\
\indent It is obvious that $V$ is open. If $V$ is not dense, without loss of generality, we can suppose that there is an open set $W$ in $U\cap \Sigma$ such that $\la\vec{F}, e_1\ra\equiv 0$ in $W$. This gives a representation for the position vector $\vec{F}$ of $\Sigma$ in $W$ as follows:
\begin{align}
    \vec{F}&=\vec{F}^{\bot}+\sum_{i=1}^{n}\la\vec{F},e_i\ra e_i; \notag\\
    &=\sum_{i=2}^{n}\la\vec{F}, e_i\ra e_i;
\end{align}
Then we can take $\{\la\vec{F}, e_2\ra, \cdots, \la\vec{F},e_n\ra\}$ as a coordinate of $W$. This leads to a contradiction since such coordinates implies we have a diffeomorphism from $W$ to an open set of $\R^{n-1}$. However $W$ is an $n$ dimensional open set. Hence $W$ is empty and $V$ is an open dense set of $U\cap \Sigma$. \\
\indent For any normal vector $n^\alpha$, $A^{\alpha}$ denotes the second fundamental form  $(h _{\alpha,ij})$, where $h _{\alpha,ij}=\la \bnb_{e_i}e_j, n^{\alpha}\ra$ and $\bnb$ is the {\Cd} of $\R^{n+m}$. We define the following sets:
$$
V_k=\{\vec{F}:\vec{F}\in V,\quad \mathrm{rank}(A^{\alpha})=k\} \quad \text{for $k=0,1,\cdots, n$}.
$$
It is easy to see that $V=V_0\cup \cup_{k=1}V_k$ and $A^{\alpha}\equiv 0$ on $V_0$. Assume we can show that $V_0$ is dense in $V$. By the continuity of $A^{\alpha}$, we will obtain that $A^{\alpha}\equiv 0$ in $V$. Since $V$ is dense in $U\cap \Sigma$, $A^{\alpha}\equiv 0$ in $U\cap \Sigma$ for any normal vector $n^\alpha$.  This gives that $\Sigma$ is totally geodesic in $U$.\\
\indent Next we show that $V_0$ is dense in $V$. Let $V_k^0$ be the interior of $V_k$ for $k\geq 1$. To prove the denseness of $V_0$ in $V$, it is suffice to prove that $V_k^0$ is empty for $k\geq 1$.\\
\indent Fix $k\geq 1$. Suppose $V_k^0$ is not empty. We emphasize it is an $n$-dimensional  manifold in $\Sigma$. Let $\vec{F}$ be any point in $V_k^0$ written as
\begin{equation}
    \vec{F}=\sum_{i=1}^n\la {F},e_i\ra e_i
    \end{equation}
    Then the map $T:V_k^0\rightarrow \R^n$ given by
    \begin{equation}
        T(\vec{F})=(\la F,e_1\ra,\cdots, \la F,e_n\ra), \vec{F}\in V_k^0
    \end{equation}
    is a local diffeomorphism on $V_k^0$.\\
 \indent    Since $\la \vec{F}, n^\alpha\ra \equiv 0$, taking the derivatives {\wrt} $\{e_i\}_{i=1}^{n}$ gives that
\be\label{eq:midterm}
T(\vec{F})A^\alpha=\sum_k\la \vec{F}, e_k\ra h_{\alpha,ik }\equiv 0;
\ene
Here we used that $\bnb_{e_i}n^{\alpha}=-h_{\alpha,ik}e_k+\sum_{\beta}\la \bnb_{e_i}n^{\alpha}, n^{\beta}\ra n^{\beta}$ and $e_i\la \vec{F}, n^\alpha\ra\equiv 0$.

Consequently $T(\vec{F})\subset ker (A^\alpha)$ which is an $(n-k)$-dimensional submanifold of $\R^n$. This leads to a contradiction. Hence $V_k^0$ is empty for $k\geq 1$. Finally $\Sigma$ is totally geodesic in $U$. The proof is complete.
\ep
As a result, we obtain the following rigidity for self-shrinkers based on an integral technique.
\bt\label{st:A:self-shrinker}  Let $\Gamma(f)$ be a self-shrinker in $\R^{n+m}$, where $f:E\subset\R^n \rightarrow \R^m$ is a smooth map with singular values satisfying $ \lambda_i\lambda_j \leq 1$ for all $i\neq j$. Suppose one of the following conditions holds:
\begin{itemize}
	\item[(i)] $E=\R^n$;
	\item [(ii)] $E=\R^n_{+}=\{(x_1,\cdots,x_{n-1},x_n):x_n>0\}$, $\sup_E |Df|\leq \mu<1$ where $\mu$ is a fixed constant, $f\in C(\bar{E})$ and $f(0)=0$,
\end{itemize}
then $\Gamma(f)$ is an $n$-dimensional plane through the origin over $E$.
\et
\br The following proof is based on the thesis of the second author in \cite[Chapter 4]{Zhou15}. The conclusion of item (ii) is much stronger than item (ii) in Theorem 4.1, because of no linear assumption on the boundary. Item (i) has been obtained in \cite[Theorem 10,(i)]{DW11} with an elliptic method.
\er
\bp For simplicity, we set
$$
\mathcal{A } :=\sum_{i,k}\lambda_i^2  h_{n+i,ik} ^2+\sum_{i,j,\alpha} h_{\alpha,ij} ^2+2\sum_{i<j,k}\lambda_i\lambda_j h_{n+i,jk}h_{n+j,ik}.
$$
Then \eqref{eq:st_ad} can be rewritten as
\begin{equation} \label{eq:estd}
\Delta  \ln(\ast\Omega) - \F{1}{2}\la\vec{F},\nb  \ln(\ast\Omega)\ra +\mathcal{A }\equiv 0.
\end{equation}
Notice that $\nb  \ln(\ast\Omega)=\F{\nb(\ast\Omega)}{\ast\Omega}$ and
\begin{align}
	|\nb\ast\Omega|^{2} 
	&=\sum_k(\sum_i\Omega(e_1,\cdots,\bnb_{e_i}e_{k},\cdots, e_n))^2 \notag\\
	&=\sum_k\big(\sum_{i,\alpha}h_{\alpha,ik}\Omega(e_1,\cdots,\underset{ i -th}{e_\alpha},\cdots, e_n)\big)^2 \label{eq:md:1} \\
	&=\sum_k(\sum_{i}\lambda_i h_{n+i,ik})^{2}(*\Omega)^2,\label{eq:md:2}
\end{align}
we have
\begin{align}
	\mathcal{A} -\F{1}{2n}|\nb   \ln(\ast\Omega) |^{2}&= \mathcal{A } -\F{1}{2n}\sum_{k}(\sum_i\lambda_i h_{n+i,ik})^{2}\notag \\
	&\geq \mathcal{A}-\F{1}{2}\sum_{i,k} \lambda_i^{2} h _{n+i,ik} ^{2}\notag\\
	&\geq \F{1}{2} \sum_{i,k}\lambda_i^2  h_{n+i,ik}^2+ \sum_{i,j,\alpha} h_{\alpha,ij} ^2+2\sum_{i<j,k}\lambda_i\lambda_jh_{n+i,jk}h_{n+j,ik};\notag\\
	&\geq \F{1}{2} \sum_{i,k}\lambda_i^2  h_{n+i,ik} ^2+\sum_{\alpha,i}h_{\alpha,ii}^{2}\notag\\
	&+\sum_{i<j,k}\big( h_{n+i,jk} ^2+ h_{n+j,ik} ^2+2\lambda_i\lambda_jh_{n+i,jk}h_{n+j,ik}\big) \label{eq:get0}\\
	&\geq 0.\notag
\end{align}
In the first inequality, we applied the Cauchy inequality  $(\sum_i\lambda_ih_{n+i,ik} )^2\leq n  \sum_{i}\lambda_i^2 h_{n+i,ik} ^2 $.
In the last inequality, we used the fact that $ \lambda_i\lambda_j  \leq 1$.

Now we use a similar technique as in \cite[Lemma 3.10]{Zhou18a-d}.  Fix $r\geq 1$, we denote by $\phi$ a compact supported smooth function in $\R^{n+m}$
such that $\phi\equiv 1$ on $B_{r}(z_r)$ and $\phi\equiv 0$ outside of $B_{r+1}(z_r)$ with
$|\nb \phi|\leq |D\phi|\leq 2$. Here $D\phi$ and $\nb \phi$ are the gradient of $\phi$ in
$\R^{n+k}$ and $\Sigma$, respectively, and $z_r$ is a point in $\R^{n+m}$ to be determined later.
Multiplying \eqref{eq:estd}  by $\phi^{2}\ex$ and integrating over $\Gamma(f)$, we obtain
\begin{align}
	0&= \int_{\Gamma(f)} \phi^{2} div_{\Gamma(f)}(\ex \nb \ln(\ast\Omega) ) + \int_{\Gamma(f)}\phi^{2}\ex \mathcal{A}\notag\\
	&= -\int_{\Gamma(f)} 2\phi\la\nb \phi, \nb \ln(\ast\Omega) \ra \ex + \int_{\Gamma(f)}\phi^{2}\ex \mathcal{A}\notag\\
	&\geq -2n\int_{\Gamma(f)}|\nb\phi|^{2}\ex +\int_{\Gamma(f)}\phi^{2}\ex(\mathcal{A}-\F{|\nb \ln(\ast\Omega) |^{2}}{2n}).\label{eq:e_d}
\end{align}
In the last step above, we used that $2|\phi\la \nb\phi, \nb \ln(\ast\Omega) \ra|\leq 2n|\nb\phi|^{2}+\phi^2\F{|\nb \ln(\ast\Omega) |^2}{2n}$. By our assumption of $\phi$, \eqref{eq:e_d} implies that
\begin{equation}\label{det:D}
\begin{aligned}
	\int_{\Gamma(f)\cap B_{r}(z_r)}\ex(\mathcal{A}-\F{|\nb \ln(\ast\Omega) |^{2}}{2n})&\leq 2n\int_{\Gamma(f)\cap B_{r+1}(z_r)}|\nb\phi|^{2}\ex\\
	&\leq C(n,m) r^{n}e^{-\F{r^2}{4}}.
\end{aligned}
\end{equation}
Here we used the fact that $\Gamma(f)$ has the polynomial growth property in both cases. In the case of item (i), this property was obtained in \cite{CZ13} and \cite{DX13}. In the case of item (ii), such property is easily derived from the assumption $|Df|\leq \mu<1$.\\
\indent In the case of item (i), we choose $z_r$ as the origin of $\R^{n+m}$. In the case of item (ii), $z_r=(0,0,\cdots, 0,\underset{n-th }{r+1}, 0,\cdots, 0)$. Letting $r$ go to infinity, \eqref{eq:get0} and \eqref{det:D} indicate that
\begin{align*}
	0&\geq \mathcal{A}-\F{|\nb  \ln(\ast\Omega)|^2}{2n}\\
	&\geq \F{1}{2} \sum_{i,k}\lambda_i^2  h_{n+i,ik} ^2+\sum_{i,k} h_{n+i,ik} ^{2}+\sum_{i<j,k}\big( h_{n+i,jk} ^2+ h_{n+j,ik} ^2+2\lambda_i\lambda_jh_{n+i,jk}h_{n+j,ik}\big).
\end{align*}
on any interior point of $\R^n$ or $\R^n_+$.
Since $ \lambda_i\lambda_j  \leq 1$ for any $i\neq j$, we conclude that
\begin{align}
		& h_{n+i,ik}=0\quad &\text{for any $i, k$;}\label{eq:est_b}\\
	& h_{n+i,jk} ^2+ h_{n+j,ik} ^2+2\lambda_i\lambda_jh_{n+i,jk}h_{n+j,ik}=0\quad &\text{for any $i,j,k$.}\label{eq:est_c}
\end{align}

Next, we claim that the mean curvature vector of $\Sigma$ vanishes, i.e.
\begin{equation}\label{eq:min}
\vec{H}= \sum_{i} (\sum_j h_{n+i,jj})e_{n+i}\equiv 0.
\end{equation}
Item (i) and Item (ii) are both concluded from Theorem \ref{thm:v_normal}.  \\
\indent Now fix any $i_0\in\{1,\cdots,m\}$, and we proceed as follows:
\begin{enumerate}
	\item By \eqref{eq:est_b}, $h_{n+i_0,i_0i_0}=0$.
	\item Fix $j\neq i_0$. \begin{enumerate}
		\item If $ \lambda_{i_0}\lambda_j \leq  1-\delta<1$ for some $\delta\in (0,1)$,  \eqref{eq:est_c} gives that
		\begin{align*}
			0=&\delta( h_{n+i_0,jk} ^2+ h_{n+j,i_0k} ^2)+(1-\delta)(h_{n+i_0,jk} ^2+ h_{n+j,i_0k} ^2)\\
			&+2\F{\lambda_{i_0}\lambda_j}{1-\delta}h_{n+i_0,jk}h_{n+j,i_0k} .
		\end{align*}
		Then $h_{n+i_0,jk}=0$ for any $k$. Let $k=j$, then we get $h_{n+i_0,jj}=0$.
		\item Otherwise, it holds that $\lambda_{i_0}\lambda_j =1$. Then $\lambda_{j}\neq 0$, \eqref{eq:est_b} implies that $h_{n+j,jk}=0$ for any $k$. In particular, $h_{n+j,ji_0}=0$.  With this fact, \eqref{eq:est_c} indicates that
		\begin{align*}
		0= h_{n+i_0,jj} ^2+ h_{n+j,i_0j}^2+2\lambda_i\lambda_jh_{n+i_0,jj}h_{n+j,i_0j}= h_{n+i_0,jj} ^2.
			\end{align*}
		Hence $h_{n+i_0,jj}=0$.
	\end{enumerate}
    \end{enumerate}
As a result, we conclude that $h_{n+i_0,jj}=0$ for any $j\neq i_0$. Consequently, $\sum_{j}h_{n+i_0,jj}=0$. Since we choose $i_0$ arbitrarily, the claim in \eqref{eq:min} holds true. The proof is complete.
\ep
\subsection{Smoothness of {\Lp} self-shrinker  graphs}\label{Sec smooth}

In this section we derive the smoothness of {\Lp} minimal-type graphs in Theorem \ref{smooth:c-minimal}. We will generalize \cite[Theorem 5.3]{Fischer80} and \cite[Theorem 4.1]{Wang2004} to the setting of {\Lp} self-shrinkers.\\
\indent The following condition will be sufficient for our purposes in the proof of Theorem \ref{main:thm:A}.
\begin{definition}\label{lp:strictly:decreasing} Let $f:E\subset \R^n\rightarrow\R^m$ be a {\Lp} map. We say that $f$ is \textit{strictly length decreasing} if  $ \lambda_i  <1-\Sc$ for any $i  \in \{1,2,\cdots, n\}$ and for some constant $\Sc\in (0,1)$, whenever $df$ is well-defined and $\lambda_i's$ are the singular values of $df$.
    \end{definition}
Now we use a more general definition including minimal graphs and self-shrinkers in Euclidean spaces.
\begin{definition} For any constant $c\geq 0$, an $n$-dimensional smooth $c$-\textit{minimal} submanifold $\Sigma$ in $\R^{n+m}$ is a stationary point of the following area functional $\mathcal{F}(\Sigma)$ given by
    \begin{equation}
          \mathcal{F}(\Sigma):=\int_{\Sigma}\exp(-c\F{|z|^2}{4})d\mathcal{H}^n\quad
        \end{equation}
where $z=(x_1,\cdots, x_n, y^1,\cdots, y^m)\in\Sigma$ and $\mathcal{H}^n$ is the $n$-dimensional Hausdorff measure.
    \end{definition}
 \begin{remark} A $0$-minimal submanifold is precisely a minimal submanifold, and a $1$-minimal submanifold is precisely a self-shrinker.
     \end{remark}
     \indent Suppose $f$ is a {\Lp} map from $E$ into $\R^m$. Then $\mathcal{F}(\Gamma(f))$ takes the form
     $$
     \mathcal{F}(\Gamma(f))=\int_{E}\exp \big(-c\F{|\vec{F}|^2}{4}\big)\sqrt{{\rm det}(I+df^T df)}dx
     $$
where   $|\vec{F}|^2=\sum_{i=1}^n x_i^2+\sum_{A=1}^m (f^A)^2$.
\begin{definition}\label{weak:solution:A}  We say that a {\Lp} function $f:E\subset \R^n\rightarrow \R^m$ is \emph{a weak solution to the $c$-minimal surface system} if it satisfies that
    \begin{equation}
    \begin{aligned}
        &\int_{E}\exp  \big( -c\F{|\vec{F}|^2}{4}\big)\bigg(  \sum_{i,j,A}\sqrt{g}g^{ij}\la \F{\P f^{A}} {\P x_i}, \F{\P \vp^{A}}{\P x_j}\ra -\F{1}{2}c\la \vec{F},\vp\ra\bigg )  dx =0
        \end{aligned}
        \end{equation}
for any  map $\vp =(\vp ^1 ,\cdots, \vp ^m )\in C_{c}^1(E;\R^m)$.
    \end{definition}
 Let $f$ be as in \eqref{weak:solution:A}. Then for each $A=1,\cdots,m$, $f^A$ is the weak solution to the following elliptic equation
 \begin{equation}\label{general:equation}
     g^{ij}\F{\P^2 f^A}{\P x_i\P x_j}+cq(f,\F{\P f^A}{\P x_1},\cdots,\F{\P f^A}{\P x_n})=0,\quad A=1,\cdots, m,
     \end{equation}
where $q$ is smooth with respect to each of its components and $(g^{ij})$ is a smooth matrix-value function depending on  $df^1,df^2,\cdots,df^m$. Then the first variation of $\mathcal{F}(\Gamma(f))$, still denoted by $\vec{H}$, satisfies that
\begin{equation}\label{first:variation}
    \vec{H}+c\F{\vec{F}^{\bot}}{2}=0.
    \end{equation}
   \indent   The following result is implicitly used in the proof of \cite[Theorem 4.1]{Wang2004}. For the convenience of the readers, we include its proof here.
     \begin{theorem}\label{lpcone} Suppose $T$ is an $n$-dimensional {\Lp} strictly length decreasing minimal cone as a graph in $\R^{n+m}$, which is smooth away from vertex. Then $T$ is a plane.
\end{theorem}
\bp In the case of $n=1$, $T$ is the union of two half lines intersecting at the vertex. It is contained in a plane in $\R^{n+m}$. The conclusion is obvious.\\
\indent Now assume $n=2$. Under our assumption, we have \eqref{star:Omega}, i.e.
\begin{equation}
	\Delta \ln *\Omega\leq -\Sc \sum_{\alpha,l,k}h^2_{\alpha,lk}
\end{equation}
on $T$ except its vertex. Since $T$ is a cone, the minimum of $*\Omega$ can be achieved at some point on the intersection of $T$ and the unit sphere in $\R^{n+m}$. By the maximum principle, $*\Omega$ is a positive constant and $\sum h^2_{\alpha,lk}=0$. Thus $T$ is totally geodesic. Since $n\geq 2$, the set $spt(T)$ with the origin removed is connected. Therefore, $spt(T)$ is contained in a $n$-dimensional linear space in $\R^{n+m}$. By the connectedness, $T$ is just this $n$-dimensional linear space.
\ep
\begin{theorem}\label{smooth:c-minimal} Suppose $f:E\subset\R^n\rightarrow \R^m$ is a {\Lp} strictly length decreasing  map  as a weak solution to the $c$-minimal surface system in \eqref{first:variation}. Then $f$ is smooth over $E$.
\end{theorem}
\bp According to equation \eqref{general:equation}, to show $f$ is smooth, it suffices that $f$ is $C^{1,\alpha}$. Because $\vec{H}+c\F{\vec{F}}{2}=0$ holds in the weak sense, $\Gamma(f)$, the graph of $f$, has locally uniformly bounded first variation \eqref{first:variation}. Therefore the tangent cone of $\Gamma(f)$, $T$, at $p$ exists. Moreover $T$ is a strictly length decreasing {\Lp}
graph as the weak solution to the minimal surface system.

We use a similar idea as  in Yuan\cite{Yuan2002}. Suppose not. Assume  $f$ is  singular at  some point $p_{0}\in E$. We can blow-up the graph of $f$ and derive a minimal cone $T$ by the monotonicity formula. Moreover, $T$ remains strictly length decreasing.

We claim that $T$ is smooth away from its vertex. Otherwise,  $T$ is singular at some $p\neq p_{0}$. We blow up $T$ at $p$ to get a lower dimensional minimal cone $T'$ cross a line $l$. If $T'$ still has singularity away from its vertex, we repeat the blow up process, and finally, we will stop with a minimal cone singular only at its vertex. Using Theorem \ref{lpcone}, $T$ is flat. This contradicts  the Allard's regularity theorem.

Hence, $T$ is smooth away from its vertex.   Using Theorem \ref{lpcone} again, we see that $T$ is a plane. Using Allard's regularity theorem,  $f$ is regular at $ p_{0}$.

Therefore, $f$ is smooth.
\ep
Combining the conclusion (2) of Theorem \ref{smooth:bernstein:theorem} with Theorem \ref{smooth:c-minimal}, yields the following result.
\bt \label{smooth:B} Let $E$ be $\R^n$ or  $\R^n_+$.  Suppose $f:E \rightarrow \R^m$ is  a {\Lp} strictly length decreasing map  as the weak solution of the minimal surface system (the case $c=0$ of  \eqref{first:variation}).  In the case of $E=\R^n_+$, $f=0$ on $\P E$. Then $f$ is a linear map over $E$.
\et
The following rigidity result on self-shrinkers follows from combining Theorem \ref{st:A:self-shrinker} and Theorem \ref{smooth:c-minimal}.
\bt\label{lp:self-shrinker} Let $E$ be $\R^n$ or  $\R^n_+$. Suppose $f:E \rightarrow \R^m$ is a {\Lp} strictly length decreasing map as the weak solution of the 1-minimal surface (self-shrinker) system. In the case $E=\R^n_+$, $f=0$ on $\P \R^n_+$. Then $f$ is a linear map over $\R^n_+$.
\et

\section{Proof of Theorem \ref{main:thm:A}}\label{Sec pf thm:A}
In this  section, we prove Theorem \ref{main:thm:A}.  We begin by establishing the short-time existence of the MCF flow \eqref{MCF3} in Lemma \ref{lm:st:existence}. Then we show that the length decreasing property  \eqref{deLA} is preserved along this flow  in Lemma \ref{first:step:MCF}. The corresponding long-time existence is given in Lemma \ref{second:step:MCF}. Finally, we conclude Theorem \ref{main:thm:A}.
\subsection{The short-time existence}
First we recall some definitions from Lieberman \cite{Lieberman1996} with a minor modification. Let $M$ be a Riemannian manifold with  metric $\sigma_M$ and let $E\subset M$ be a smooth bounded domain. We denote $d:M\times M\rightarrow[0,\infty]$ as the distance function and denote $\mathcal{E} =E\times \R $. For any two points $X=(x,t),Y=(y,s)\in \mathcal{E}$, the distance between $X$ and $Y$ in  $\mathcal{E}$ is given by
\begin{equation}
    |X-Y|:=\max\{d(x,y), |t-s|^{\F{1}{2}}\}.
    \end{equation}
  Suppose $f:W\subset \mathcal{E}\rightarrow \mathbb{R}$. For any $\alpha\in (0,1]$ and $\beta\in (0, 2]$, we define
\begin{align*}
    \text [f]_{\alpha}(X):=\sup_{Y\in W \setminus\{X\}  }\frac{|f(Y)-f(X)|}{ |X-Y|^{\alpha}}; \quad \langle f\rangle_{\beta}(X):=\sup_{Y=(y,s)\in W\setminus  \{X\} }\frac{|f(Y)-f(X)|}{ |s-t|^{\beta/2}}.
\end{align*}
For any $X=(x,t)\in W$, 
  any integer $l\geq 0$ and any $\alpha \in (0,1]$, 
set
\begin{gather*}
    |D_{x}^lD_t^jf|(X):=\sup_{\mathcal{V}_1,\cdots,\mathcal{V}_l}|D^l_{\mathcal{V}_1,\cdots,\mathcal{V}_l}D_t^jf|(X);\\
    [D_{x}^lD_t^jf]_\alpha(X):=\sup_{\mathcal{V}_1,\cdots,\mathcal{V}_l}[D^l_{\mathcal{V}_1,\cdots,\mathcal{V}_l}D_t^jf]_\alpha(X);\\
      \langle D_{x}^lD_t^jf\rangle_{\alpha+1}(X):=\sup_{\mathcal{V}_1,\cdots,\mathcal{V}_l}\langle D^l_{\mathcal{V}_1,\cdots,\mathcal{V}_l}D_t^jf\rangle_{\alpha+1}(X),
    \end{gather*}
where $\{\mathcal{V}_1,\cdots, \mathcal{V}_l\}$ ranges over all unit vector fields near $x$. Define
\begin{align*}
 |f|_{k+\alpha }(X):=&  \sum_{l+2j\leq k } | D _{x}^{l}D_{t} ^{j}f|(X)+\sum_{l+2j=k} [D_{x}^lD_t^jf]_{\alpha }(X)\\
 &+ \sum_{l+2j=k-1}       \langle D^l_{x}D^j_tf      \rangle_{\alpha+1}(X),
 \end{align*}
 and $|f|_{k+\alpha,W}:=\sup_{X\in W}|f|_{k+\alpha}(X)$.

 Let  $f=(f^1,\cdots, f^m):W \subset \mathcal{E}\rightarrow \R^m$($m\geq 2 $), we say that $f\in H_{k+\alpha}(W)$ if $|f|_{k+\alpha,W}:=\sum_{A=1}^m |f^A|_{k+\alpha}(W)<\infty$.

For the short-time existence of the MCF in \eqref{MCF3}, the parabolic system \eqref{MCF} admits an analysis analogous to that of a scalar parabolic equation.
\bl \label{lm:st:existence} Suppose $\psi:\bar{E}\rightarrow \R^m$ is smooth. Then there exists $\Sc>0$ such that the mean curvature flow in \eqref{MCF3} exists smoothly on $[0,\Sc)$.
    \el
     The strategy is to generalize the proof of Lieberman \cite[Theorem 8.2]{Lieberman1996} in the setting of \eqref{MCF3} on Riemannian manifolds.
    \begin{proof}
     We give a sketch of the proof on Riemannian manifolds for completeness.
Fix $\theta\in (0,1)$ and denote $V=\bar{E}\times [0,\Sc)$, where $\Sc$ is a positive constant to be determined later.  Set $\zeta:=  |\psi|_{1+\theta, V}+1 $, here $\psi(x,t)\equiv \psi(x)$. Consider the set
    \begin{equation}\label{de:At}
        \mathcal{F}:=\{ v(x,t)\in H_{1+\theta}(V): |v|_{1+\theta,V} \leq \zeta \},
    \end{equation}
    It is obvious that $\psi\in \mathcal{F}$.

    Now define an operator $J:  \mathcal{F}\rightarrow H_{1+\theta}(V)$ by setting  $Jv=u$ if, in any local coordinate $\{x_{1}, \ldots,x_{n}\}$ on $E$, it holds that
     \begin{equation} \label{def:AB}
   \left\{\begin{aligned}
    L^A(v)[u^{A}] &=0,\quad A=1\cdots,m\quad{\text{where}}\quad L^A(v): =-\P_t  +g^{ij}(v) \partial_{ij};
    \\ u(x,0)&=\psi(x),\ x\in E; \quad u(x,t)=\psi(x),\  x \in \P E,
    \end{aligned}\right.
     \end{equation}
     where $g_{ij}(v)  =\delta_{ij} +\sum _{A}v_{i}^{A}v_{j}^{A}$ and  $(g^{ij}(v)) =(g_{ij}(v)) ^{-1}$.\\
    \indent Since $v\in H_{1+\theta}(V)$ is fixed, $ I\geq (g^{ij})\geq \frac{1}{C} I$, where $I$ is the identity matrix and $C$ is a positive constant depending  only  on $v$. Thus \eqref{def:AB} is actually a system of m Cauchy-Dirichlet problems with bounded $C^\theta(E)$ coefficients.
    The Perron process for \eqref{def:AB} for parabolic systems,  \cite[Chapter III, Section 4]{Lieberman1996}), and the comparison principle, remain valid on Riemannian manifolds. Therefore, the conclusion of \cite[Chapter V, Theorem 5.14]{Lieberman1996} still holds for \eqref{def:AB} on $E$. Namely, for each $A=1,\cdots, m$, there exists a unique solution $u^A\in H_{2+\alpha}  (\bar{E}\times [0,\Sc))$ to \eqref{def:AB}, for some $\alpha\in (0,1)$. Moreover for $u=(u^1,\cdots, u^A)$, it holds that
    \begin{equation*}
    	|u|_{H_{1+\theta}} (V)\leq C(\zeta)|\psi|_{1+\theta}\leq C(\zeta).
    \end{equation*}
   By the Newton-Leibniz formula,
    \begin{equation*}
     |u-\psi|_{H_{2+\alpha}(V)}   \leq C( \zeta)\Sc.
    \end{equation*}
When $\psi\equiv 0$, this implies that $J$ is continuous.
One can choose $\Sc$ small enough such that $u$ is contained in $\mathcal{F}$, and thus $J$ is a map from $\mathcal{F}$ to $\mathcal{F}$. Notice that $\mathcal{F}$ is a compact and convex set. Applying the Schauder fixed point theorem \cite[Theorem 8.1]{Lieberman1996}, there exists a fixed point $u\in \mathcal{F}$ such that $Ju=u$. Thus we obtain the short-time existence of the MCF flow \eqref{MCF3}.
\end{proof}
\subsection{The long-time existence}
	Suppose the MCF $F(x,t)=(x,f_{t}(x))$ in \eqref{MCF3} exists smoothly on $E\times[0,T)$.
 \bl\label{first:step:MCF} Suppose the assumption \eqref{condition:A} holds. Then there exists $\Sc \in (0,1)$, such that
\begin{equation}\label{deLA}
    \sup_{E}|Df_t|\leq 1-\Sc
\end{equation}
for any $t\in [0, T)$, where $\Sc $  depends only on $\psi$ and the assumption \eqref{condition:A}.
\el
\begin{proof}  Since $\sup_{\bar{E}}|D\psi|<1$, by continuity, there is a maximal time $T_0\leq T$ such that $\sup_{E}|Df_t|<1 $ on $[0,T_0)$. Taking $\mu=1$ and applying Proposition \ref{keyLboundary:estimate}, we derive
\begin{equation}
    \sup_{\P E}|Df_t|\leq \frac{ \omega(\psi)}{\delta} +  |D\psi |+32 n \delta   |D^{2} \psi |<1-\Sc
\end{equation}
for any fixed $\delta\in (0, \delta_0]$ satisfying \eqref{condition:A} and $t\in [0,T_0)$. Here $\Sc$ depends  only on assumption \eqref{condition:A}. In particular, for notation abuse, $\sup_E|D\psi|<1-\Sc$.
By Proposition \ref{prop:preserved}, this implies that  $\sup_{E}|Df_t|<1-\Sc$ for any $t\in [0,T_0)$. Since $T_0$ is maximal, it follows that
$T_0=T$.
\end{proof}
\rm The long-time existence of the MCF flow \eqref{MCF3} under the assumption \eqref{condition:A} follows from Lemma \ref{first:step:MCF} and Theorem \ref{lp:self-shrinker}.
\bl\label{second:step:MCF}Suppose the assumption \eqref{condition:A} holds. Then the graphic mean curvature flow $F(x,t)$ in \eqref{MCF3} can be extended over time $T$.
\el
\begin{proof}
The main tools we will use are the local regularity theory for mean curvature flow, established by White  \cite{White2005,White21}. This technique was first employed by Wang \cite{Wang2002} to show long-time existence for the mean curvature flow of closed graphs in arbitrary codimension. We adopt Wang's method with a modification to tackle the same problem for the MCF flow in \eqref{MCF3} with boundary.

 Assume $E$ is isometrically embedded in to  Euclidean space $\R^{n_0}$ for some sufficiently large $n_0$. Fix a space-time point $(Y,T)\in (\R^{n_0+m},\R_+)$. Define the backward heat kernel by
\begin{equation}
    \rho_{(Y,T)}(y,t):=\frac{1}{ \big(4\pi(T-t)\big) ^{\frac{n}{2}} }\exp\bigg( -\F{|y-Y|^2}{4(T-t)}\bigg),\quad y\in \R^{n_0}\times  \R^m, t<T,
    \end{equation}
here $n$ is the dimension of $F(x,t)$ as a submanifold in $\R^{n_0+m}$.
Define the Gaussian density of the MCF flow $F(x,t)$ in \eqref{MCF3} by
\begin{equation}\label{existence:Gauss:density}
    \Theta(Y,T):=\lim_{t\rightarrow T}\int_{  F_t(E)}\phi(|y-Y|)\rho_{Y,T}(y,t)\rm d  \mathcal{H}^n,
    \end{equation}
    where $F_t(E)$ denotes the graph of $f_t(x)$ over $E$, $\mathcal{H}^n$ denotes the $n$-dimensional \rm{Hausdorff} measure, and $\phi(t)\in C_{c}^{\infty}[0,1]$  with $\phi\equiv 1$ in $[0,\F{1}{2}]$ and $\phi'\leq 0$. The existence in \eqref{existence:Gauss:density} follows from \cite[Corollary 6.2]{White21}, since according to \eqref{deLA}, the area of $F_t(E)$ is uniformly bounded,
    \begin{equation} \label{del:estimates}
       \mathcal{H}^n (F_t(E))\leq \int_{E}\sqrt{(1+\lambda_1^2)\cdots (1+\lambda_n^2)}  d \mathcal{H}^n\leq  2^{\F{n}{2}}\mathcal{H}^n(E),
        \end{equation}
       for any $t\in [0, T)$. It is not hard to see that $\Theta(Y,T)$ is independent of the choice of $\phi$ and the isometric embedding of $E$ .

        Let $Y=(u,v)$ be any space limit of $F(x,t)$ in $\R^{n_0+m}$ as $t\rightarrow T$, where $u\in \bar{E}$ and $v\in \R^m$. To compute $\Theta(Y,T)$, we consider the parabolic dilation of space-time space $\R^{n_0}\times \R^m \times [0,+\infty)$, given by
       \begin{equation}\label{parab dila}
         \mathcal{D} _{\iota}(y,t)=(     \iota(y-Y),\iota^2(t-T))=(\iota(x-u), \iota (f_t(x)-v), \iota^2(t-T)).
           \end{equation}
  Let $S$ be the total space of the mean curvature flow $F_t(E)$ for all $t\in [0, T)$. Denote by $S^{\iota}$  the image $  \mathcal{D}_\iota(S)\subset \R^{n_0+m}\times [-\iota^2T, 0]$. By \eqref{del:estimates} and  \cite[Section 10]{White21} , for any sequence $\{\iota_i\}_{i=1}^\infty$ converging to $\infty$, there exists a subsequence, still denoted by $\{\iota_i\}_{i=1}^\infty$, such that $S^{\iota_i}$ converges to a tangent flow $S^\infty\subset \R^{n_0+m}\times ( -\infty, 0]$. Moreover, $S^\infty$ is invariant under $  \mathcal{D}_\iota$, i.e. $S^\infty$ is  self-similar. For each $t<0$, $S^\infty(t)$ denotes the space slice of the tangent flow $S^\infty$. Furthermore, $S^{\infty}(t)= \sqrt{-t} S^{\infty }(-1)$, where $S^{\infty}(-1)$ is given by the graph of a function $f_\infty$,  which is a {\Lp} weak-solution to the $1$-minimal surface system(  self-shrinker \eqref{first:variation}) over a domain $E'$. \\
  \indent From Lemma \ref{first:step:MCF}, each $f_t$ is strictly length decreasing. Notice that the property of strictly length decreasing is preserved under the parabolic dilation \eqref{parab dila}. Hence, $f_\infty$ is {\Lp}  and strictly length decreasing.

   If $Y$ lies in the interior of $E\times \R^{m}$, then $E'$ should be $T_u E$ viewed as $\R^n$. If $Y$ lies on the boundary of $E\times \R^{m}$, then $E'$ is the half-space $\R^n_+ $. By Theorem \ref{lp:self-shrinker}, in both cases, $f_{\infty}$ must be a linear map. Therefore in the first case $S^{\infty}(-1)$ is an $n$-dimensional linear submanifold in $\R^{n_0+m}$ through $0$. In the second case $S^{\infty}(-1)$ is a half space of an $n$-dimensional linear submanifold in $\R^{n_0+m}$. As a result, in the first case, $\Theta(Y,T)=1$ and in the second case, $\Theta(Y,T)=\F{1}{2}$. By \cite{White2005}, the mean curvature flow $F(x,t)$ can be extended smoothly over time $T$ in both cases.
  \end{proof}
\subsection{The limit} Combining Lemma \ref{first:step:MCF} and Lemma \ref{second:step:MCF} together, we conclude that,  under the assumption \eqref{condition:A}, the mean curvature flow in \eqref{MCF3} $F(x,t)$ exists for all time. Moreover $F(x,t)=(x, f(x,t))$, where $f(x,t)$ satisfies $\sup_{E}|Df_{t}|<1-\Sc$ for some constant $\Sc\in (0,1)$ independent of $t$. Therefore, $f(x,t)$ is strictly length decreasing for all $t>0$. \\
\indent  In the closed manifold setting, Wang \cite{Wang2002} showed that the norm of the second fundamental form of $F(x,t)$ decays exponentially as $t\rightarrow \infty$. In our setting such a strong uniform $C^2$ estimate cannot be obtained.\\
\indent In the following, let $\Sigma_t$ denote the graph of $f(x,t)$ for fixed $t$. In fact,  as the area element $\sqrt{g}$ satisfies $ \frac{d\sqrt{g}}{dt}=-|H|^{2}\sqrt{g}$, we see that the quantity  $\int_{	0}^{\infty}\int_{\Sigma_{t}}|H|^{2} d\mu dt$ is bounded by integration. Here $d\mu$ is the induced volume form. So there exists a sequence $\Sigma_{t_{i}}$ such that $\int_{\Sigma_{t_{i}}}|H|^{2}d\mu\rightarrow0$ as $ t_{i}\rightarrow\infty$. \\
\indent Together with the fact that $\sup_{E}|Df_{t_{i}}|<1-\Sc$, up to taking a subsequence, we have that $\Sigma_{t_{i}}\rightarrow\Sigma$ in the sense of varifold such that $\Sigma=(x,f(x))$ with $ \int _{\Sigma}|H|^{2}d\mu=0$, and $\sup_{E}|Df|\leq 1-\Sc$ for some Lipschitz map $f:\bar{E}\rightarrow\R^{m}$ with $f=\psi$ on $\P E$. Namely, $f$ is strictly length decreasing as a weak solution to the minimal surface system. By \cite[Theorem 4.1]{Wang2004} or Theorem \ref{smooth:c-minimal}, $f$ is smooth on $E$.\\
\indent Consequently, we obtain the desirable solution to the minimal surface system \eqref{DP:MSS} under the assumption \eqref{condition:A}. The proof is complete.
	\section{Proof of Theorem \ref{main:thm:B}}\label{Sec  pf thm:B}
The proof of Theorem \ref{main:thm:B} is divided into two parts: the existence part,  Section \ref{sec exist}, and the asymptotic behavior part,  Section \ref{sec asympt}. We emphasize that the boundary $\partial E$ may admit multiple components.

\subsection{The existence}\label{sec exist}
  Without loss of generality, suppose the origin $0$ is in $\R^{n}\setminus \bar{E}$. As a result the distance between the origin and $\P E$,   $d_0=d(0,\P E)>0$. Let $\mathrm B_r(0)\subset\R^{n}$ be the open ball  centered at the origin $0$ with radius $r$. Denote $Q_r:= E \cap\mathrm B_r(0)$. Then, for any $r>2d_0$, the boundary of $Q_r$ consists of two disjoint parts: $\P B_{r}(0)$ and $\P E$.

We claim that if $r$ is sufficiently large, we can choose $\delta_{0}=\delta_{0}(n,\partial E)>0$, which is independent of $r$ to guarantee that our conclusion holds.

Firstly, we fix $\eta_0>0$ such that the distance function $d(x)=d(x,\P E)$  is $C^{2}$ on  $E_{\eta_0}=\{x\in E: d(x,\P E)<\eta_0\}$. Suppose  $\F{r}{2}> \F{r_{0}}{2}:=  diam(\P E)+2\eta_0+d_0 $. Then  $E_{\eta_0}\subset B_{\frac{r}{2}}(0)$. Now the set $Q_{r,\eta_0}=\{x\in Q_{r}: d(x,\P Q_{r})<\eta_0\}$ consists of two disjoint parts: $E_{\eta_0}$ and $E'_{\eta_0} =\{x\in Q_{r}: d(x,\P B_r(0))<\eta_0\}$. Then the function $ \tilde{d}(x) :=d(x,\P Q_{r})$ on $Q_{r,\eta_0}$ satisfies
\begin{equation}
   \tilde{d}( x)=\left\{
\begin{aligned}
   &d(x),\quad \quad x\in E_{\eta_0};\\
   &d(x,\P B_{r}(0)), \quad  x\in E'_{\eta_0},
    \end{aligned}\right.
    \end{equation}
   due to the fact that set $E'_{\eta_0}$ is disjoint with $B_{\frac{r}{2}}(0)$.

  Similarly, we compute $-\Delta_{\Sigma_t} \tilde{d}(x)$ on $Q_{r}$ as in the proof of  Lemma \ref{lap}, where $\Sigma_t$ is the mean curvature flow in $Q_{r}\times \R^m$ in which the initial data is the graph of $\psi$. 
  Since $\P B_{r}(0)$ is strictly convex, arguing as in Remark \ref{delta0rmk} gives
    \begin{equation}
        -\Delta_{\Sigma_t}  \tilde{d}=-\Delta_{\Sigma_t}d(x,\P B_{r}(0))\geq 0 \quad \text{on}\  E'_{\eta_0}
       \end{equation}
  On the other hand, from Lemma \ref{lap}, there exists a constant $c_{0}=c_0(n,\partial E)>0$ which is independent of $r>r_0$, such that
   \begin{equation}
      -\Delta_{\Sigma_t}  \tilde{d}=-\Delta_{\Sigma_t} d\geq -c_0 \quad \text{on}\  E_{\eta_0}.
  \end{equation}
  In both cases, we have
  \begin{equation}
      -\Delta_{\Sigma_t} \tilde{d}\geq -c_0 \quad \text{on}\ Q_{r,\eta_0}.
      \end{equation}
By Remark \ref{remark:definitioneta}, the claim follows. 

Now we can prove verbatim as in Proposition \ref{keyLboundary:estimate} that there exists  $\tilde{\delta}_{0}=\tilde{ \delta}_{0}(n,\partial  E )>0$ such that for any $\delta\in (0,\tilde{\delta}_0)$,
\begin{align*}
    |Df_{t}(x)| \leq &    \frac{ w(\psi)}{\delta} +  |D\psi |+16n(1+\mu) \delta   |D^{2} \psi |, \quad \mbox{on} ~\partial Q_r \times [0,T),
\end{align*}
where $\mu=\sup_{Q_r \times [0,T)}|Df |^{2}$ and $\Sigma_t=(x,f_t(x))$ is the mean curvature flow on time interval $[0,T)$.

Arguing as in the proof of Theorem \ref{main:thm:A}, for any $\delta\in (0, \tilde{\delta}_0)$ and $\psi$ satisfying assumption \eqref{condition:B},  i.e.
\begin{equation}\label{condition:D}
    \F{w
        (\psi)}{\delta}+\sup
    _{E}|D\psi|+32 n\delta\sup_{E}|D^2\psi|<1-c,
\end{equation}
 the Dirichlet problem of the MSS \eqref{DP:MSS} admits a solution $f_r\in C^\infty(Q_r)\cap C(\bar{Q}_r)$ solving
\begin{equation}\label{DP:MSS:Q}
    \left\{
    \begin{aligned}
        g^{ij}(f_{r})(f_{r})^A_{ij}&=0\quad A=1,2,\cdots,m\quad\text{on} ~Q_r,\\
        f _{r}&=\psi \text{\quad on~}\P Q_r.
    \end{aligned}
    \right.
\end{equation}
Moreover, $\max_{Q_r}|Df_r|<1-c$. Also by the maximum principle, $\max_{Q_r}|f_r|\leq \max_{\bar{Q}_r}|\psi|$.\\
\indent By the Arzel\`{a}–Ascoli theorem, there exists a subsequence   $\{f_{r_i}\}_{i=1}^\infty$ converging to a {\Lp} function $f$ as the weak-solution to the following minimal surface system:
\begin{equation}
	\left\{
	\begin{aligned}
		g^{ij}(f ) f^A_{ij}&=0\quad A=1,2,\cdots,m\quad\text{on} ~E,\\
		f  &=\psi \text{\quad on~}\P E.
	\end{aligned}
	\right.
\end{equation}
Moreover $\sup_{E}|D f|\leq 1-c$ on $E$. Now $f$ is {\Lp} and strictly length decreasing on $E$. By \cite[Theorem 4.1]{Wang2004} or Theorem \ref{smooth:B}, $f$ is smooth on $E$. This gives the existence of Theorem \ref{main:thm:B}.

\subsection{The asymptotic behavior}\label{sec asympt}
We consider the asymptotic behavior of the solution $f$ to the MSS \eqref{DP:MSS} over $E$. Notice that the singular values  satisfy $\sup_{E}|Df|\leq 1-c$.
Now we blow down $\Gamma(f)$, the graph of $f$, following the idea of Simon \cite{Simon1983}. Namely consider the function $f_{k}(x):=k f(k^{-1}x)$ defined over $kE =\{y\in \R^n:ky\in E\}$. Let $\Gamma(f_{k})$ denote the graph of $f_k$ over $kE$.\\
\indent  Applying the classical compactness of stationary varifolds, there exists a subsequence $k_i\rightarrow\infty$ as $ i\rightarrow\infty$,  such that the sequence $\Gamma(f_{k_{i}})$ converges to a minimal cone $T$. Moreover $T$ is a graph of a strictly length decreasing Lipschitz function $f_\infty$, with $|Df_\infty|<1-c$ as the weak solution to the minimal surface system. By Theorem \ref{smooth:B}, $f_{\infty}$ is a linear map on $\R^{n}$. This implies that $\sup_{|x|\rightarrow +\infty}|Df-l|=0$, where $Df$ means the gradient of $f$. The proof is complete.

		\bibliographystyle{alpha}
		\bibliography{Ref-hc-Drichlet-problem}

	\end{document}